\documentclass[a4paper, 10pt]{article}

\usepackage{amssymb}
\usepackage{amsmath}
\usepackage{mathrsfs}
\usepackage{amsfonts}
\usepackage{color}
\usepackage{vmargin}
\usepackage{amsthm}

\usepackage{a4}
\usepackage{graphicx}

\long\def\symbolfootnote[#1]#2{\begingroup%
\def\thefootnote{\fnsymbol{footnote}}\footnote[#1]{#2}\endgroup}

\setmarginsrb{20mm}{20mm}{20mm}{20mm}{10mm}{10mm}{10mm}{10mm}

\newcommand{\Om}{\Omega}

\newcommand{\R}{\mathbb{R}}

\newcommand{\Hei}{{\mathbb H}_{1}}

\newcommand{\HW}{{HW}^{1,p}}

\newcommand{\HWloc}{{HW}^{1,p}_{loc}}

\newcommand{\HWzero}{{HW}^{1,p}_{0}}

\newcommand{\modq}{{\rm Mod}_Q}
\newcommand{\modp}{{\rm Mod}_p}
\newcommand{\Mod}{{\rm Mod}}
\newcommand{\bd}{\partial}
\newcommand{\dist}{{\rm dist}}
\newcommand{\diam}{{\rm diam}}
\newcommand{\bdyP}{\partial_{\rm P}}
\newcommand{\diamH}{\diam_{\,\Hei}}
\newcommand{\distH}{\dist_{\,\Hei}}
\newcommand{\dH}{d_{\,\Hei}}
\newcommand{\ga}{\gamma}
\newcommand{\fv}{f^{-1}}

\definecolor{blau}{rgb}{0.1,0.0,0.9}

\definecolor{funk}{rgb}{0.1,0.4,0.9}

\newcounter{komcounter}
\numberwithin{komcounter}{section}

\newcommand{\dbd}[2]{\frac{\partial#1}{\partial #2}}

\def\XXint#1#2#3{{\setbox0=\hbox{$#1{#2#3}{\int}$}
     \vcenter{\hbox{$#2#3$}}\kern-.5\wd0}}

\theoremstyle{plain}
\newtheorem{theorem}{Theorem}[section]
\newtheorem{thm}[theorem]{Theorem}
\newtheorem{lem}{Lemma}[section]
\newtheorem{cor}{Corollary}[section]
\newtheorem{ex}{Example}
\newtheorem{observ}{Observation}[section]

\theoremstyle{definition}
\newtheorem{defn}{Definition}[section]
\newtheorem{rem}{\textnormal{\textbf{Remark}}}

\begin{document}

\title {Prime ends in the Heisenberg group $\Hei$ and the boundary behavior of quasiconformal mappings}

\author{
Tomasz Adamowicz{\small{$^1$}}
\\
\it\small Institute of Mathematics, Polish Academy of Sciences \\
\it\small ul. \'Sniadeckich 8, 00-656 Warsaw, Poland\/{\rm ;}
\it\small T.Adamowicz@impan.pl
\\
\\
Ben Warhurst{\small{$^1$}}
\\
\it\small Institute of Mathematics,
\it\small University of Warsaw,\\
\it\small ul.Banacha 2, 02-097 Warsaw, Poland\/{\rm ;}
\it\small B.Warhurst@mimuw.edu.pl
}

\date{}
\maketitle

\footnotetext[1]{T. Adamowicz and B. Warhurst were supported by a grant Iuventus Plus of the Ministry of Science and Higher Education of the Republic of Poland, Nr 0009/IP3/2015/73.}

\begin{abstract}
We investigate prime ends in the Heisenberg group $\Hei$ extending
N\"akki's construction for collared domains in Euclidean spaces. The
corresponding class of domains is defined via uniform domains and
the Loewner property. Using prime ends we show the counterpart of
Caratheodory's extension theorem for quasiconformal mappings, the
Koebe theorem on arcwise limits, the Lindel\"of theorem for principal points and the Tsuji theorem.
\end{abstract}

\section{Introduction}

The corner stone for the theory of prime ends is a work by
Carath\'eodory~\cite{car1}, who first defined prime ends for
simply-connected domains in the plane. The main motivation for his
studies came from the problem of continuous and homeomorphic
extensions of conformal mappings. A result due to Carth\'eodory (and
Osgood--Taylor) allows for the homeomorphic extension of conformal mappings
between Jordan domains in the plane. However, there are simple
examples, for instance a slit-disk, when this extension theorem
fails. Nevertheless, by introducing the so-called prime ends
boundary, Carath\'eodory was able to show that \emph{a conformal
homeomorphism between bounded simply-connected planar domains $U$
and $V$ extends to a homeomorphism between $U$ and the prime ends
compactification of $V$}. The subsequent development of the prime
ends theory has led to generalizations of prime ends for more
general domains in the plane and in higher dimensional Euclidean
spaces, to mention Kaufman~\cite{Kau}, Mazurkiewicz~\cite{Maz},
Freudenthal~\cite{Fre} and more recently Epstein~\cite{Ep} and
Karmazin~\cite{Ka}, see also~\cite{abbs} for a theory of prime ends
in metric spaces. Applications of prime ends encompass: the theory
of continua, see Carmona--Pommerenke~\cite{cp1, cp2}, the boundary
behavior of solutions to elliptic PDEs, see Ancona~\cite{An} and the
studies of the Dirichlet problem for $p$-harmonic functions in
metric spaces, see Bj\"orn--Bj\"orn--Shanmugalingam~\cite{bbsh}.

In this work we follow the original motivation for studying prime
ends and investigate extension problems and the related boundary
behavior for quasiconformal mappings in the setting of the Heisenberg
group $\Hei$. Similar results of this type were obtained by
V\"ais\"al\"a~\cite[chapter 17]{va1}, \cite{va2} and N\"akki~\cite{na} in the Euclidean setting. The latter one introduced prime ends based on the notion of the $n$-modulus of curve families in $\R^n$. One of our goals is to
generalize N\"akki's results to the sub-Riemannian setting. If one seeks
to explore these ideas in other geometric settings then the
Heisenberg group $\Hei$ together with the sub-Riemannian geometry
is a natural candidate. The reason being that $\Hei$ has a large
enough family of quasiconformal mappings to make it an interesting
pursuit, see the discussion in the end of section~\ref{prel-quasic}. It is perhaps surprising that such a generalization is not straightforward and requires some new approaches. First we recall some basic definitions for the Heisenberg including rectifiable curves, contact and quasiconformal mappings, which we define also in terms of the modulus of curve families (the rudimentary properties of modulus in $\Hei$ are recalled and proved in the Appendix).

In Section~\ref{sec-Nakkis-pe} we introduce prime ends which we define following the approach in \cite{na}. Upon introducing a topology on the prime ends boundary, Definition~\ref{defn-topol}, we show our first extension result, allowing us to extend a quasiconformal mapping to a homeomorphism between the prime ends boundaries, Theorem~\ref{thm-hom-1}. One of the most important definitions of our work are given in Sections~\ref{sect-loew-unif} and \ref{sect-collared}. There we recall the Loewner spaces, uniform domains in $\Hei$ and observe in Lemma~\ref{sep-cond} that in uniform domains our modulus-based definition of prime ends has an equivalent form in terms of the Heisenberg distance. This result is the key-part of our Definition~\ref{def-right-collared} of the so-called collared domains. The original definition introduced by V\"ais\"al\"a and N\"akki cannot be applied directly in our setting due to the rigidity of the conformal mappings in $\Hei$ and the lack of the domains satisfying the Loewner condition (problems which do not arise in the Euclidean setting). Furthermore, in Section~\ref{sect-collared} we relate collaredness with another important class of domains finitely connected at the boundary and prove that \emph{a quasiconformal map from a collared domain $\Om$ has a homeomorphic extension to a map between a topological closure $\overline{\Om}$ and the prime ends closure of the target domain}, see Theorem~\ref{thm-key-res}. This result naturally corresponds to Theorem 4.1 in \cite{na} and Section 3.1 in \cite{va2}.

The goal of Section~\ref{sect4} is to present yet another perspective on prime ends and note that in the domains finitely (in particular, locally) connected at the boundary, one can construct singleton prime ends associated with every boundary point. We also relate our prime ends to those studied in \cite{abbs} in metric spaces.

The important results of this paper are presented in
section~\ref{sec-boundary}, where we study the boundary behavior of
quasiconformal mappings. We first recall notions of accessibility
and observe that one can assign to an accessible boundary point the
singleton prime end, see Observation~\ref{obs-acc-pe}. All together,
in the presentation below we propose three methods to obtain
canonical prime ends in $\Hei$: by employing collardness
(Observation~\ref{obs-coll-spe}), via the finite connectedness at
the boundary (Lemma~\ref{obs-finite-pe}) and in
Observation~\ref{obs-acc-pe}. We show the Koebe theorem providing
conditions which imply that a quasiconformal mapping has arcwise
limits along all end-cuts in domains finitely connected at the
boundary (Theorem~\ref{thm-Koebe}). This result corresponds to the
classical observation for conformal mappings and generalizes similar
result in $\R^n$ due to N\"akki~\cite[Theorem 7.2]{na}. Then we
prove a version of the Lindel\"of theorem relating the principal
points of prime ends to cluster set of mappings along end-cuts
(Theorem~\ref{thm-subs}). The proof of this result requires
developing some new observations and illustrates differences between
the Euclidean and the Heisenberg settings. The corresponding results
in $\R^n$ are due to Gehring~\cite[Theorem 6]{gehr1},
N\"akki~\cite[Theorem 7.4]{na} and Vuorinen~\cite[Section 3]{vuo}.
Finally, in Theorem~\ref{tsuji-thm} we show a variant of the Tsuji
theorem on the Sobolev capacities of sets of arcwise limits. In the
proof we face once again the lack of some techniques available in
$\R^n$, namely the modulus symmetry property.

\section{Preliminaries}

 In this section we recall basic definitions and properties of the Heisenberg groups, including brief discussion on curves and their lengths, the Heisenberg and the sub-Riemannian metrics. Moreover, we recall notions of the horizontal Sobolev spaces and quasiregular and quasiconformal mappings in $\Hei$. Further discussion, including the definition and properties of the modulus of curve families, is presented in the Appendix.

\subsection{The Heisenberg group $\Hei$}\label{sec-hei-prelim}

The Heisenberg group is often presented using coordinates $(z,t)$ where $z=x+iy \in \mathbb{C}$, $t \in \mathbb{R}$ and multiplication is defined by
\begin{align}
(z_1,t_1)(z_2,t_2)&=(z_1+z_2, t_1+t_2 + 2\, {\rm Im}\,(z_1 \bar z_2)) \nonumber\\
&=(x_1 + x_2 , y_1 + y_2, t_1+t_2 + 2(x_2y_1-x_1y_2)). \label{hz}
\end{align}
Note that $(z,t)^{-1}=(-z,-t)$.

In particular a natural basis for the left invariant vector fields is given by the following vector fields
\begin{align}
\tilde X = \dbd{}{x}+2y\dbd{}{t}, \quad \tilde Y =\dbd{}{y}-2x\dbd{}{t} \quad {\rm and } \quad  \tilde T= \dbd{}{t},
\end{align} where $[\tilde X , \tilde Y ]=-4 \tilde T$. The horizontal bundle is given pointwise by $\mathcal{H}_p={\rm span}\, \{ \tilde X(p), \tilde Y(p)\}$ and a curve is horizontal if for almost all $t_0 \in I$, $\gamma'(t_0)$ exists and belongs to $  \mathcal{H}_{\gamma(t_0)}$.

The pseudonorm given by
\begin{align}
||(z,t)|| =  ( |z|^4 + t^2)^{1/4} \label{HN1}
\end{align}
gives rise to a left invariant distance defined by $d_{\Hei}(p,q)= ||p^{-1}q||$ which we call the Heisenberg distance. More explicitly we have
\begin{align*}
d_{\Hei} ( (z_1,t_1), (z_2,t_2)) & = ||(-z_1,-t_1)(z_2,t_2)|| \\
 & =||(z_2-z_1, t_2-t_1 -2 {\rm Im} (z_1 \bar z_2))|| \\
& =(   |z_2-z_1|^4 + (t_2-t_1 -2 {\rm Im} (z_1 \bar z_2))^2 )^{1/4}.
\end{align*}

A dilation by $r \in \mathbb{R}$ is defined by $\delta_r (z,t)  = ( rz, r^2 t)$, indeed  $d_{\Hei}(\delta_r (p),\delta_r (q))=|r| d_{\Hei}(p,q)$. The left invariant Haar measure $\lambda$ is simply the $3$-dimensional Lebesgue measure on $\Hei$ and $\delta_r^*d \lambda=r^4 d \lambda$. It follows that the Hausdorff dimension of the metric measure space is $(\Hei, d_{\Hei}, \lambda)$ is $Q=4$. An equivalent statement is that $(\Hei, d_{\Hei}, \lambda)$ is $Q=4$ Ahlfors--regular which is to say that there exists a real constant $c$ such that for all balls $B(r,p)$ we have
\begin{align}
\frac{1}{c}r^Q  \leq \mathcal{H}^Q(B(r,p)) \leq c r^Q, \label{alfR}
\end{align}  where $\mathcal{H}^Q$ denotes  $Q$-dimensional Hausdorff measure induced by $d_{\Hei}$.

\begin{defn}
A $Q$-regular metric measure space will be a triple $(X,d,\mu)$ where the Hausdorff dimension of $(X,d)$ is $Q$ and $\mu$ is a constant multiple of the $Q$-dimensional Hausdorff measure induced by $d$.
\end{defn}

Examples are when $X$ is a Carnot group with sub-Riemannian distance $d_s$ and Haar measure. Indeed, the Haar measure is a multiple of Lebesgue measure which is a  multiple of the $Q$-dimensional Hausdorff measure induced by $d_s$. We can replace $d_s$ with any equivalent metric. On $\Hei$  in particular the measure $\mathcal{H}^Q$ is a constant multiple of $3$-dimensional Lebesgue measure, an inequality similar to \eqref{alfR} is valid with $\mathcal{H}^Q$ replaced by $\lambda$.

\subsection{Rectifiable curves}

 A curve $\gamma$ in $\Hei$ is a continuous map $\gamma:I \to \Hei$ where $I$ is an open or closed interval. If $I=[a,b]$ then the Heisenberg length of $\gamma$ is given by $$ l(\gamma)=\sup \sum_{i=1}^n d_{\Hei}(\gamma(t_i),\gamma(t_{i+1})),$$
where the supremum is over all finite sequences $a=t_1 \leq t_2 \leq \dots \leq t_n \leq t_{n+1}=b$. If $I$ is not closed then
 $$ l(\gamma)=\sup l( \gamma|_J )$$ where the supremum is over all closed subintervals $J \subset I$. If   $l(\gamma)<\infty$ we say that $\gamma$ is rectifiable.

A curve  $\gamma:I \to \Hei$ is locally rectifiable if each subcurve $\gamma|_{[\alpha,\beta] }$ is rectifiable for all closed intervals $[\alpha,\beta] \subseteq I$. For example the curve $\gamma:(-1,1) \to \Hei$ defined by $\gamma(t)=(t+iy(t),0)$ where
$$
y(t)=\begin{cases} t\sin(1/t) &\mbox{if } t \ne 0  \\
0  & \mbox{if } t=0, \end{cases}
$$
is not locally rectifiable since any subcurve $\gamma(t)|_{[\alpha, \beta]}$ such that $0 \in [\alpha,\beta]$ is not rectifiable. Conversely the curve  $\gamma:(0,1) \to \Hei$ defined by $\gamma(t)=(t+iy(t), 0)$ is locally rectifiable but not rectifiable, see Chapter 3 in~\cite{va1}.

The following theorem is proved in exactly the same way as Theorem 3.2 in \cite{va1} with the  Euclidean metric replaced by the Heisenberg distance and so we omit the proof.

\begin{thm}\label{curve-ext} If $\gamma:(a,b) \to \Hei$ is rectifiable then it has a unique extension $\gamma^*:[a,b]  \to \Hei$ such that $l(\gamma^*)=l(\gamma)$.
\end{thm}

For each rectifiable curve $\ga$ of a closed interval there is a unique arc length parametrization of $\gamma$ arising from the arc length function  $S_\gamma:[a,b] \to [0, l(\gamma)]$ given by $S_\gamma(t)=l(\gamma|_{[a,t]})$. In particular there is a unique $1$-Lipschitz map $\bar \gamma: [0, l(\gamma)] \to \Hei$ called the arc length parametrisation such that $\gamma(t)=\bar \gamma \circ S_\gamma(t)$. The arc length parametrisation facilitates the definition of the line integral of a nonnegative Borel function $\varrho:\Hei \to [0,\infty]$ as follows:
\begin{align}
 \int_{\gamma} \varrho dl:=\int_{0}^{l(\gamma)} \varrho \circ \bar \gamma (s)\, ds. \label{LineInt}
\end{align}
If $I$ is open, then we set
\[
 \int_{\gamma} \varrho dl:= \sup_{\gamma'} \int_{\gamma'} \varrho dl.
\]
where the supremum is over all closed subcurves $\gamma'$.

By Pansu \cite{pansu}, Lipschitz mappings are Pansu differentiable almost everywhere. For  locally rectifiable curves this means that $\lim_{t \to 0} \delta_{1/s} \circ \tau_{\bar \gamma(s_0)}^{-1} \circ \bar \gamma(s_0+s)$ exists for almost all $s_0$ which implies that $\bar \gamma'(s_0)$ exists and is horizontal, moreover the same is true for  $\gamma'(s_0)$. If $\gamma :[a,b] \to \mathbb{H}$  is a horizontal curve, then the sub-Riemannian length of $\gamma$ is given by the integral
$$
 l_S(\gamma)= \int_a^b \sqrt{ \dot x (s)^2 + \dot y(s)^2} \, ds
$$
and $l(\gamma)=l_S(\gamma)$ (see Koran\'yi~\cite{kor85}).  Moreover the change variable $s=S_\gamma(t)$ in \eqref{LineInt} shows that
$$\int_\gamma \varrho  d l =  \int_a^b \varrho(\gamma(s)) |\gamma '(s)|   ds \quad $$ where $| \gamma' (s)| =\sqrt{ \dot x (s)^2 + \dot y(s)^2}$.

The sub-Riemannian distance $d_S(p,q)$ is defined as the infimum of sub-Riemannian lengths of all horizontal curves joining $p$ and $q$. The Heisenberg metric and the sub-Riemannian metric are equivalent, to be precise
$$ \frac{1}{\sqrt{\pi}} d_S(p, q) \leq d_{\Hei}(p,q)  \leq d_S(p, q), $$
see Bella\"iche~\cite{bel}.

\subsection{Horizontal Sobolev space on $\Hei$, contact, quasiregular and quasiconformal mappings on $\Hei$}\label{prel-quasic}

Below we recall some basic definitions in the theory of the Sobolev spaces in $\Hei$ and contact mappings, and apply them to define the main classes of mappings we study in the paper, namely quasiregular and quasiconformal mappings.

\begin{defn}\label{horiz-Sob}
 Let $U\subset \Hei$ be an open subset of $\Hei$. For $1<p<\infty$, we say that a function $u:U\to \R$ belongs to the \emph{horizontal Sobolev space} $\HW(U)$ if $u\in L^p(U)$ and the horizontal derivatives $\tilde X u$ and $\tilde Y u$ exist in the distributional sense and represented by elements of $L^p(U)$. The space $\HW(U)$ is a Banach space with respect to the norm
\[
 \|u\|_{\HW(U)}\,=\,\|u\|_{L^p(U)}+\|(\tilde Xu, \tilde Yu)\|_{L^p(U)}.
\]

In the similar way we define the local spaces $\HWloc (U)$. We define space $\HWzero(U)$ as a closure of
$C_{0}^{\infty}(U)$ in $\HW(U)$.
\end{defn}

The horizontal gradient  $\nabla_0 u$ of $u \in \HWloc (U)$ is given
by the equation
 $$\nabla_0 u =  (\tilde X u)  \tilde X  + (\tilde Y u)  \tilde Y.$$

A contact form on $\Hei$ is given by $ \omega= dt + 2(x dy - y dx )$ , in particular $\omega \wedge d\omega$ is a volume form and $\mathcal{H}_p={\rm ker}\, \omega_p$.
\begin{defn}
 Let $\Om,\Om'\subset \Hei$ be domains in $\Hei$. We say that a diffeomorphism $f:\Om\to\Om'$ is a \emph{contact transformation} if it preserves the contact structure, i.e.
\begin{equation}
 f^{*}\omega = \lambda \omega,
\end{equation}
where $\lambda:\Om \to \R$ and $\lambda\not= 0$ in $\Om$.
\end{defn}
Note that the definition implies $f_*$ preserves the horizontal bundle, moreover we we can weaken the regularity assumption to $HW_{loc}^{1,Q}(\Omega, \Omega')$ and simply require statements to hold $\lambda$-a.e. The contact maps that will be of relevance in our work here will be those which are quasiconformal.

Let $f:\Omega \to \Omega'$ be a homeomorphism where $\Omega$ and $\Omega'$ are domains in $\Hei$, and let the distortion function of $f$ be given by $$H_f(p,r)=\frac{\sup\{d(f(p),f(q))\, |\, d(p,q) \leq r \} }{ \inf \{d(f(p),f(q))\, |\, d(p,q) \geq r \}}.$$
\begin{defn}\label{qc-def-H}
Let $\Om, \Om'\subset \Hei$ be domains in $\Hei$. A homeomorphism $f:\Om\to\Om'$ is \emph{$K$-quasiconformal} if
$$\limsup_{r \to 0} H_f(p,r) \leq K $$ for all $p \in \Omega$.
\end{defn}

Quasiconformal maps are Pansu differentiable, see \cite{pansu}, which implies that they are  $HW_{loc}^{1,Q}(\Omega, \Omega')$ regular contact maps.  We recall that  the Pansu differential $Df(p)$ is the automorphism of $\Hei$ defined as
$$
Df(p)\,q := \lim_{t\rightarrow 0} \delta_t^{-1}\circ \tau_{f(p)}^{-1}\circ f \circ \tau_p\circ \delta_t(q),
$$
where $p,q\in \Hei$. It follows that quasiconformality can be expressed analytically by the inequality
$$ ||Hf_*||_\infty^4 \leq K \det f_*$$
where
$$ ||Hf_*||_\infty=\max \{ | f_*(V) | \, : \, V \in \mathcal{H}_p, \,  |V|=\sqrt{dx(V)^2+dy(V)^2}=1 \}.$$

A fundamental property of quasiconformal mappings is the fact that they are absolutely continuous on almost all locally rectifiable curves in the sense that the family $\Gamma_f$ consisting of rectifiable curves whose image under a quasiconformal map $f$ is not rectifiable satisfies $\Mod_4(\Gamma_f)=0$ (see Appendix for definitions and some properties of the modulus of curve families). See Theorem 9.8 in Heinonen--Koskela--Shanmugalingma--Tyson~\cite{hkstB} for a proof in the setting of spaces with locally bounded geometry. Stated in terms specific to the Heisenberg setting we have:

\begin{thm}\label{qc-def-mod}{\cite[Thm 18]{BalFassPlat}}
If $f: \Omega \to \Omega'$ is quasiconformal map between two domains in $\Hei$  and $\Gamma$ is a curve family in $\Omega$, then
\begin{equation}\label{qc-mod}
 \frac{1}{K^2}\Mod_4(\Gamma)\leq  \Mod_4(f\Gamma ) \leq K^2 \Mod_4(\Gamma).
\end{equation}
\end{thm}

In Capogna--Cowling~\cite{CapCow}, the authors prove that $1$-quasiconformal maps are $C^\infty$ and, consequently from Koran\'yi--Reimann~\cite{kr1}, the following Liouville theorem holds: a $1$-quasiconformal map of a domain $\Omega \subseteq \Hei$ is given by the action of an element in $SU(1,1)$. In particular, a $1$-quasiconformal map is always a composition of the following four basic types of $1$-qc map:
\begin{enumerate}
\item Left translation (isometry),
\item Dilation ($1$-qc),
\item Rotation: $R_\theta(z,t) = (e^{i \theta} z, t)$ (isometry),
\item Inversion in the unit sphere : $J(z,t)= \frac{-1}{|z|^4 +t^2 } ( z(|z|^2 +it) , t)=  ( \frac{z}{it -|z|^2  }, \frac{-t}{|z|^4 +t^2 })$ ($1$-qc).
\end{enumerate}

We remark that via the Cayley transformation, the inversion is understood in terms of the one point compactification of $\Hei$ being the unit sphere in $\mathbb{C}^2$, see \cite{kr1}. The inversion facilitates the definition of stereographic projection of any sphere to the complex plane. Using translations and dilation, the given sphere is mapped to the sphere with center $(0,-3/2)$ and radius $1/\sqrt{2}$ and then inverted in the unit sphere centered at $(0,-1)$, i.e.,  apply
$$
 \tau_{(0,-1)} \circ  J \circ \tau_{(0,1)}.
$$
 We note that, unlike the case of $\overline{ \mathbb{R}^n}$, we do not have the freedom to normalize, since left translations do not preserve the complex plane.

Since the $1$-qc maps are given by the action of finite dimensional Lie group we say that $\Hei$ is $1$-qc rigid. In such cases a Carath\'eodory extension theorem  for $1$-qc mappings is somewhat trivial. Similarly, if we are going to consider a non-trivial Carath\'eodory extension theorem for quasiconformal maps we at least need to avoid Carnot groups that are contact rigid, i.e., the contact maps are given by the actions of a finite dimensional Lie group, see Ottazzi--Warhurst~\cite{OW}. Following Euclidean space, the most nonrigid of all Carnot groups is $\Hei$. Indeed, the pseudo group of local contact mappings is large and so a reasonably interesting theory can be expected. In fact, in \cite{kr1} they produce an infinite dimensional family of quasiconformal maps as flows of vector fields as well as developing a Beltrami type equation. However, there is no existence theorem for this equation. On the other hand, in Balogh~\cite{Zolt}, it is shown that quasiconformal maps exist on $\Hei$ that are not bi-Lipschitz.

\section{Prime ends in the Heisenberg group $\Hei$}

 In this section we give basic definitions of the prime ends theory in the sub-Riemannian setting. First, following the modulus approach of N\"akki, we define prime ends and a topology on the prime ends boundary. Using prime ends, we show the first extension result for quasiconformal mappings, see Theorem~\ref{thm-hom-1}. The remaining part of this section is devoted to study the so-called \emph{collared domains}. N\"akki~\cite{na} and V\"ais\"al\"a~\cite{va1} defined collared domains in order to study extension properties and the prime end boundary. It turns out, that the structure of the Heisenberg group does not allow us to follow their approach. Namely, the Loewner property of collaring domains, crucial for the properties of prime ends, need not hold for natural counterparts of collaring domains in $\Hei$. Therefore, we need new definition, in particular we impose additional uniformity assumption on the collaring neighborhood. See details in section~\ref{sect-collared} and section~\ref{sect-loew-unif} for Loewner and uniform domains in the Heisenberg setting. Using collared domains we obtain Theorem~\ref{thm-key-res}, another extension result for quasiconformal mappings.

\subsection{Prime ends according to N\"akki}\label{sec-Nakkis-pe}

N\"akki in \cite{na} introduced a theory of prime ends for domains in $\R^n$ based on the notion of $n$-modulus.
We follow his idea and develop the appropriate theory in the Heisenberg setting based on the notion of $Q$-modulus where $Q=4$ is the Hausdorff dimension of $\Hei$.

\begin{defn}[cf. Section 3.1 in \cite{na}]\label{def-cross-set}
A connected subset $E$ of a domain $\Om\subset \Hei$ is called a \emph{cross-set} if:
\begin{itemize}
\item[(1)] $E$ is relatively closed in $\Om$,
\item[(2)] $\overline{E}\cap \partial \Om\not=\emptyset$,
\item[(3)] $\Om\setminus E$ consists of two components whose boundaries intersect $\bd \Om$.
\end{itemize}
\end{defn}

\begin{defn}\label{def-chain-N}
A collection $\{E_k\}_{k=1}^\infty$ of cross-sets is called a \emph{chain} if $E_k$ separates $E_{k-1}$ and $E_{k+1}$ within $\Om$ for all $k$. We denote the component of $\Om\setminus E_k$ containing $E_{k+1}$ by $D(E_k)$ and define an \emph{impression} of a chain $\{E_k\}_{k=1}^\infty$ as follows
\[
I[E_k]:=\bigcap_{k=1}^{\infty}\overline{D(E_k)}.
\]
\end{defn}
The definition immediately implies that impression of a chain is either a continuum or a point. The set of all chains is in some sense too large so the following additional conditions are imposed to cut it down.

\begin{defn}\label{def-pr-chain}
A chain is a \emph{prime chain} if:
\begin{itemize}
\item [(a)] $\Mod_{4}(E_{k+1}, E_{k}, \Om)<\infty$,
\item [(b)] For any continuum $F \subset \Om$ we have that
$$\lim_{k\to \infty} \Mod_{4}(E_k, F, \Om)=0.$$
\end{itemize}
\end{defn}

In view of Theorem \ref{qc-def-mod}, conditions (a) and (b) are  quasiconformally invariant, and under certain restrictions on $\Omega$, imply stronger separation of the cross sets as well as control over their diameter. In particular we will discuss domains $\Om$ so that (a) implies $\dist_{\Hei} (E_k,E_{k+1})>0$ and (b) implies $\diam_{\Hei}(E_k) \to 0$, cf. Lemma~\ref{sep-cond}.

It is crucial for our further work to know that the impression of a (prime) chain is a subset of the topological boundary $\partial \Om$. This follows from Lemma A.10 in~\cite{abbs} together with Part (b) of Definition~\ref{def-pr-chain}. Lemma A.10 is formulated for the so-called acceptable sets (cf. Definition~\ref{def-accset} below) but for the sake of convenience we will state it without appealing to the definition of prime ends as in \cite{abbs} and specialize to the setting of $\Hei$. This is due to the fact that $\Hei$ satisfies the main assumptions of \cite{abbs}, that is $\Hei$ is a complete metric measure space with a doubling measure, supporting a $(1,4)$-Poincar\'e inequality.

\begin{lem}[Lemma A.10 in \cite{abbs}]\label{lem-imp-bd}
Let $\{E_k\}_{k=1}^\infty$ be a sequence of open bounded connected sets $E_k\subsetneq\Omega$ such that $\overline{E_k}\cap \partial \Omega\not=\emptyset$ satisfying $\overline{E}_{k+1}\cap\Om\subset E_k$ for each $k$. If
$\lim_{k\to\infty}\Mod_4(E_k,B,\Om)=0$ for some ball $B\subset\Om\setminus E_1$, then $I=\bigcap_{k=1}^\infty \overline{E}_k\subset\bd\Om$.
\end{lem}

Let now $\{E_k\}_{k=1}^\infty$ be a (prime) chain in a domain $\Om\subset \Hei$. Notice that sets $D(E_k)$, as in Definition~\ref{def-chain-N}, for all $k$ satisfy assumptions of Lemma~\ref{lem-imp-bd}. Moreover, $\Gamma(E_k,B,\Om)<\Gamma(D(E_k),B,\Om)$ for all $k$, and thus, by Part 5 of Lemma~\ref{lem-mod-prop} implies that property (b) of Definition~\ref{def-pr-chain} holds for $\{D(E_k)\}_{k=1}^{\infty}$ and any continuum $F\subset \Om$ as well. If $B$ is a ball with $B\subset \Om\setminus D(E_1)$, then $F=\overline{B}$ is a continuum in $\Om$. Furthermore, $\Mod_4(D(E_k),B,\Om)\leq \Mod_4(D(E_k), F,\Om)$. Thus, Lemma~\ref{lem-imp-bd} implies that $I:=\bigcap_{k=1}^\infty \overline{D(E_k)}\subset\bd\Om$.

It turns out that one can define an equivalence relation on the set of prime chains in a given domain. This give rise to one of the main notions of our work, the so-called prime ends.

\begin{defn}\label{def-pr-end}
Two chains $\{E_k\}_{k=1}^\infty$ and $\{F_k\}_{k=1}^\infty$ are said to be equivalent if each domain $D(E_k)$ contains all but a finite number of the cross-sets $F_l$ and each domain $D(F_l)$ contains all but a finite number of the cross-sets $E_k$. The equivalence classes are called \emph{prime ends} of $\Omega$ and the set of all prime ends will be denoted $\bdyP \Omega$ and called the \emph{prime ends boundary}. We use the notation $[E_k]$ to denote the prime end defined by the prime chain $\{E_k\}_{k=1}^\infty$.
\end{defn}

If $[E_k]\in \bdyP \Om$, then the impression of any representative of $[E_k]$ is the same, and so the impression $I[E_k]$ of $[E_k]$ is well defined. By Theorem~\ref{qc-def-mod}, a quasiconformal map $f :\Om \to \Om'$, naturally extends to the prime ends by setting $f([E_k])=[f(E_k)]$.

We introduce a topology on the prime end boundary of a domain in $\Hei$. Similar construction in $\R^n$ is presented in \cite{na}, see also \cite[Section 8]{abbs} for a discussion in metric spaces. We then apply this topology in studying the extension of a quasiconformal map to a map between prime ends closures of underlying domains.

\begin{defn}\label{defn-topol}
A topology on  $\Om \cup \bdyP \Om$ is given by extending the relative toplogy of $\Omega$ by defining neighborhoods of prime ends as follows: A neighborhood of a prime end $[E_k] \in \bdyP \Om$ will have the form $U \cup U_P $ where
\begin{itemize}
\item[(a)] $U \subset \Om$ is open
\item[(b)] $\partial U \cap \partial \Om \ne \emptyset$
\item[(c)] $U \cup (\partial U \cap \partial \Om)= \tilde U \cap \Omega$ where $\tilde U$ is open.
\item[(d)] $D(E_k) \subset U$ for $k$ sufficiently large
\item[(e)] $U_P = \{ [F_l] \in \partial_P \Omega\,  :\, D(F_l) \subset U \, \, {\rm for} \, \, {\rm all } \, \, l \, \, {\rm sufficiently} \, \, {\rm large} \}$.
\end{itemize}
\end{defn}

We comment that another definition of a topology can be given if one defines the convergence of points and prime ends to a prime end. The above definition is similar in construction to the one given in Proposition 8.5 in \cite{abbs}, however, there, the constructed topology fails to be Hausdorff, cf. Example 8.9 in \cite{abbs}.

The topology as in Definition~\ref{defn-topol} is Hausdorff: It is clear that interior points of $\Omega$ are separated and points in the interior of $\Omega$ are separated from points in the prime end boundary $\bdyP \Om$. It remains to see that any pair of distinct points in  $\bdyP \Om$ are separated. To this end let $[E_j], [F_k] \in \bdyP \Om$ be distinct prime ends, then it follows that there exists $n \in \mathbb{N}$ such that $D(E_j) \cap D(F_k) = \emptyset$ for all $j,k \geq n$. If $U=D(E_n)$ and $V=D(F_n)$, then $(U \cup U_P) \cap (V \cup V_P) =\emptyset $, hence $[E_j]$ and $[F_k]$ are separated and the topology is indeed Hausdorff as claimed.

An important question to consider is: When does $\Omega \cup \partial_P \Omega$ together with the topology described above become a compact space? Obviously, it is necessary that $\Omega$ is relatively compact in the metric topology of $\Hei$ but delicate issues can arise with regards to $\partial_P \Omega$. In particular, if $\Omega$ is relatively compact and $\{U^\alpha\}$ is a covering of $\Omega$ by relatively open sets, then we can select a finite collection $\{U^\beta\}$ which covers $\Omega$ and write $\{U^\beta\}=\{U^{\beta_0}\}\cup \{U^{\beta_1}\}$ where each element of the collection $\{U^{\beta_1}\}$ satisfies $\partial U^{\beta_1} \cap \partial \Omega \ne \emptyset$. Then, one considers if $ \{ U^{\beta_0} \} \cup \{ U^{\beta_1} \cup U_P^{\beta_1} \}$ is a cover of  $\Omega \cup \partial_P \Omega$ which requires that every prime end of $\Omega$ belongs to $U_P^{\beta_1}$ for some $\beta_1$. If it is the case that all the prime ends have singleton impressions, then the this requirement is fulfilled. In the context of extension of quasiconformal maps, the domains of interest, the so-called collared domains, will be seen to have the property that all the prime ends have singleton impressions (see Section~\ref{sect-collared}).

Below we present our first extension result allowing us to extend a quasiconformal mapping between domains in $\Hei$ to a homeomorphism between the prime ends closures.

\begin{thm}\label{thm-hom-1}
Let $\Om$ and $\Om'$ be domains in $\Hei$ and let $f:\Om\to\Om'$ be a quasiconformal map of $\Om$ onto $\Om'$. The extended map $ F:\Om \cup \partial_P \Omega \to \Om' \cup \partial_P \Omega' $, where
\begin{equation*}
 F (p)=
 \begin{cases}
 f(p) & \hbox{if } p \in \Om \\
 [f(E_k)] & \hbox{if } p=[E_k] \in \partial_P \Om,
 \end{cases}
\end{equation*}
is a homeomorphism.
\end{thm}
\begin{proof}
The map $F$ is well defined. Indeed, for $p\in \Om$ it follows from $f$ being homeomorphism. For $p=[E_k]\in \bdyP \Om$, the discussion following Definition~\ref{def-pr-end} gives us that the value of $F([E_k])$ is independent on the representative of $[E_k]$.

The extended map is a bijection. If $[F_l'] \in \partial_P \Om'$ and $F_l =f^{-1}(F_l')$ for all $l$, then $F([F_l])=[F_l']$ and that $\{F_l\}_{l=1}^{\infty}$ defines a (prime) chain and, thus, a prime end in $\bdyP \Om$, follows from $f$ being a homeomorphism and Theorem~\ref{qc-def-mod}. If $F([E_k])=F([F_l])$, then by the definition of $F$ it holds $[f(E_k)]=[f(F_l)]$ which implies, again by Theorem~\ref{qc-def-mod}, that $[E_k]=[F_l]$.

  Map $F$ is continuous.  If $V \cup V_P$ is a neighborhood contained in  $\Om'\cup \partial_P \Omega'$ such that $f([E_k]) \in V_P$ for some $[E_k]\in \bdyP \Om$, then $D(f(E_k)) \subset V$ for $k$ sufficiently large, and since $D(f(E_k))=f(D(E_k))$, we have that $D(E_k) \subset f^{-1}(V)$. It follows that the preimage $F^{-1}(V\cup V_P)$ is contained in $f^{-1}(V) \cup {f^{-1}(V)}_P$. Moreover, if $[F_l] \in f^{-1}(V)_P$ then $D(F_l) \subset  f^{-1}(V)$ for $l$ sufficiently large and $ f(D(F_l))=D(f(F_l)) \subset  V$. Hence $[f(F_l)] \in V_P$ and we conclude that the preimage
  \[
  F^{-1}(V\cup V_P)=f^{-1}(V) \cup f^{-1}(V)_P,
  \]
and hence is open implying that $F$ is continuous.

The extended map is open. Let $U \cup U_P$ be a neighborhood in $\Om \cup \partial_P \Om$ and let $[E_k] \in U_P$. It follows that $F([E_k]) \in f(U)_P$ since $f(D(E_k)) =D(f(E_k))$. Furthermore, if $[F_l] \in f(U)_P$ then $ f^{-1}(D(F_l)) \subset U$ for $l$ sufficiently large. Hence  $F(U \cup U_P)=f(U) \cup f(U)_P$.
\end{proof}

\subsection{The Loewner condition and uniform domains}\label{sect-loew-unif}
  Let $(X, d, \mu)$ be a rectifiably connected metric measure space of Hausdorff dimension $Q$ equipped with a locally finite Borel regular measure. Following Definition 8.1 in  chapter 8  of Heinonen~\cite{hein}, we define a \emph{Loewner function} $\Psi_{X}:(0,\infty) \to [0,\infty)$ by the formula
\begin{equation}\label{def-Loewn1}
 \Psi_{X,p}(t) :=\inf \{ \Mod_Q \Gamma(E, F, X)\,:\,\Delta(E,F) \leq t \},
\end{equation}
where $E, F \subset X$ are nondegenerate disjoint continua in $X$ and
$$\Delta(E, F)=\frac{\dist(E, F)}{\min\{\diam E, \diam F\}}$$
denotes the relative distance between sets $E$ and $F$.

We note that by definition $\Psi_{X,p}$ is decreasing.

\begin{defn}\label{def-Loewn}
 A rectifiably connected metric measure space $(X,d,\mu)$ is said to be $Q$-Loewner if $\Psi_{X}(t)>0$  for all $t>0$.
\end{defn}

In general, if a metric measure space $(X,d,\mu)$ satisfies some connectivity and volume growth conditions then it is $Q$-Loewner if and only if it supports a weak $(1,Q)$-Poincar\'e inequality (see Chapter 9 in \cite{hein}). A Carnot group $G$ equipped with the sub-Riemannian metric $d_s$ and Lebesgue measure is such a metric measure space and in fact a $Q$-Loewner space where $Q$ is the Hausdorff dimension of $(G,d_s)$ (see Proposition 11.17 in \cite{HajKosk}).

From the point of view of our studies where we use the Heisenberg metric in preference to the sub-Riemannian, we have the following result:

\begin{thm}(Heinonen--Koskela--Shanmugalingam--Tyson~\cite[Prop 14.2.9]{hkstB})\label{loew-heis}
The metric measure space\,\,\, $(\Hei, d_{\Hei}, \lambda)$ is $4$-Loewner.
\end{thm}
We remark that the previous theorem is also a consequence of Theorem 9.27 in \cite{hein}. Next, we recall a definition a uniform domains. Such domains play an important role in analysis and PDEs, see
Heinonen~\cite{hein}, Martio--Sarvas~\cite{MS}, N\"akki--V\"ais\"al\"a~\cite{nv91} and V\"ais\"al\"a~\cite{vaisala88} for more information on uniform domains. Examples of uniform domains encompass quasidisks, bounded Lipschitz domains and some domains with fractal boundaries such as the von Koch snowflake, see also Capogna--Garofalo~\cite{cg}, Capogna--Tang~\cite{ct}.

\begin{defn}\label{def-unifDom}
  A domain $\Omega \subset \Hei$ is called uniform, if there exists two positive constants $\alpha$ and $\beta$ such that each pair of points $x, y \in \Omega$ can be joined by a rectifiable curve $\gamma$ such that:
 \begin{itemize}
\item[(a)] $l(\gamma) \leq \beta d_{\Hei}(x, y)$
\item[(b)] $\alpha \min \{ l (\gamma_{xz}), l(\gamma_{yz})\} \leq  {\rm dist}_{\Hei} (z, \partial \Om)$ for all $z \in \gamma$, where $\gamma_{xz}$ ($\gamma_{yz}$) denote subarcs of $\gamma$ joining $x$ and $z$ ($y$ and $z$).
\end{itemize}
\end{defn}

An important example of uniform domains in $\Hei$ is provided by the following result.

\begin{lem}(Capogna--Garofalo~\cite[Cor 1]{cg})\label{lem-cap-tang}
 Balls in the Heisenberg metric are uniform domains.
\end{lem}

We note that the class of uniform domains is independent of the choice between the Heisenberg metric $d_{\Hei}$ or the sub-Riemannian metric $d_S$. Indeed, this is clear once it is observed that the length of a curve in the sub-Riemannian metric is the same as the length in the Heisenberg metric, and that the two metrics are equivalent. Similarly, the class of quasiconformal maps is the same regardless of which metric we use. We can, thus, state the following theorem for the Heisenberg metric even though in Capogna--Tang~\cite{ct} it is proved only for the sub-Riemannian metric (cf. Theorem 2.15 in Martio--Sarvas~\cite{MS} for the prototypical result in the Euclidean setting).

\begin{thm}{\cite[Thm 3.1]{ct}} Let $\Om \subset \Hei$ be a uniform domain with constants $\alpha$ and $\beta$. If $f : \Hei \to \Hei$  is a global $K$-quasiconformal map, then $f (\Om )$ is a uniform domain with constants $\alpha'$ and $\beta'$ depending on $\alpha$, $\beta$, $K$ and the homogeneous dimension $Q=4$.
\end{thm}

The following theorem uses the concept of a space being locally $Q$-Loewner which is somewhat technical in its definition and of no concern anywhere else in the discusssion, so we direct the reader to Bonk--Heinonen--Koskela~\cite{BonkHeinKosk} and Herron~\cite{herr} for details rather than provide them here.

\begin{thm}{\cite[Thm 6.47]{BonkHeinKosk}} \label{thm-uni-Loewn}
 An open connected subset $\Omega$ of a locally compact $Q$-regular $Q$-Loewner space is locally $Q$-Loewner. In particular, uniform subdomains of such spaces are $Q$-Loewner.
\end{thm}

The following consequences of Theorem~\ref{thm-uni-Loewn} will be of vital importance from the point of view of the notation of collardness and prime ends, cf. Definition~\ref{def-right-collared}.

\begin{thm} {\cite[Fact 2.12]{herr}} \label{herr} Let $\Omega$ be a uniform subdomain of a locally compact $Q$-regular $Q$-Loewner space and let  $E$ and $F$ be nondegenerate connected subsets of $\overline \Omega$ with $\overline E \cap \overline F \ne \emptyset$, then $\Mod_Q(E,F, \Omega)= \infty$.
\end{thm}

\begin{lem} \label{sep-cond} If $\{E_k\}$ is a prime chain in a uniform subdomain $\Omega \subset \Hei$, then conditions (a) and (b) in Definition~\ref{def-pr-chain} become, respectively:
\begin{itemize}
\item [(a)] $\dist_{\Hei}(E_k,E_{k+1})>0$ for $k=1,2,\ldots$,
\item [(b)] $\lim_{k\to \infty}\diamH(E_k)=0$.
\end{itemize}
\end{lem}

\begin{proof}
 Property (a) is an immediate consequence of Theorem \ref{herr}. To see that (b) holds we first note that condition (b) in Definition~\ref{def-pr-chain}, and $\Delta(E_k, F) \leq t_0$ for some $t_0>0$ and all $k$, together imply that $\Psi_{X}(t_0)=0$. This contradicts the Loewner condition, and hence $\Delta(E_k,F) \to \infty$ which in turn implies (b).
\end{proof}

\subsection{Collared domains}\label{sect-collared}

 The notion of collared domains in $\R^n$ was introduced by V\"ais\"al\"a, see Definition 17.5 in \cite{va1} in the context of the boundary behavior of quasiconformal mappings, see also N\"akki~\cite{na1, na2, na3}. Moreover, in \cite{na}, N\"akki employed collaredness in the studies of prime ends based on the conformal modulus in Euclidean domains. In this section we introduce collared domains in Heisenberg setting. This notion will be subsequently used to develop prime ends theory in $\Hei$.

In some sense  V\"ais\"al\"a's definition of a collared domain is a manifold with boundary where the coordinate maps of charts containing boundary components are quasiconformal with target in the closed upper half space. In the spirit of V\"ais\"al\"a we arrive at the following definition:

\begin{defn}\label{def-collared} A domain $\Omega\subset \Hei$ is said to be locally quasiconformally collared at $x \in \partial \Om$ if there exists a neighborhood $U\subset \Hei$ of $x$ such that $U \cap \Omega$ is uniform and there exists a homeomorphism $h$ of $U \cap \bar \Omega$ onto a half ball $B(x_0,r)_+$ such that $h(x)=x_0 \in \mathbb{C}$ and  $h|_{U \cap \Omega}$ is quasiconformal.  We call $(U,h)$ a {\it collaring} coordinate and note that $U \cap \partial \Omega$ is mapped onto the open disc $ B(x_0,r)_+ \cap \mathbb{C}$.
\end{defn}

Our definition differs from Definition 17.5 in \cite{va1} in the following ways: firstly we do not require $x_0=0$ so as to ensure that balls satisfy the definition via stereographic projection, in particular we cannot follow up with a normalisation to $x_0=0$ with left translation since left multiplication does not stabilize $\mathbb{C}$,  and secondly and perhaps most importantly, we impose the assumption of some local uniformity which is required so that Lemma \ref{sep-cond} is applicable. However, there is one major drawback in this definition, namely, we need to know that half balls in $\Hei$ are uniform which is made difficult to verify since half balls are not mapped to half balls under left translations while the definition of ball is a left translation of the ball at the origin.

Moreover, there is another problem with mimicking \cite{va1}. Namely, in Theorem 17.10, and in \cite{na} Lemma 2.3 it is proved that in the Euclidean case, collaring coordinates satisfy a local Loewner property without assuming local uniformity. However their proof relies on estimates involving the modulus of curve families contained in spheres (\cite[Sec 10]{va1}), which are not available to us in the Heisenberg setting, due to the fact that spheres contain very few horizontal curves.

Therefore, we propose the following approach to collardness.

\begin{defn}\label{def-right-collared} A domain $\Omega \subset \Hei$ is said to be locally quasiconformally collared at $x \in \partial \Om$ if there exists a uniform subset $U \subseteq \Omega $ and a quasiconformal map $h:U \to B(0,1)$ such that $h(U)=B(0,1)$ and:
\begin{itemize}
\item[(a)] $x \in \partial U$.
\item[(b)] $h$ extends homeomorphically to a map $h:\partial U \cap \partial \Om\to \partial B(0,1)$ such that $h(\partial U \cap \partial \Om) \subseteq \partial B(0,1)$ is connected, closed, and contains $h(x)$ in the interior of $\partial B(0,1)$ (in the topology of $\partial B(0,1)$).
\end{itemize}
We call $(U,h)$ a {\it collaring} coordinate of $x$.
\end{defn}

\begin{ex}\label{ex-ball-coll}
 A ball $B$ in $\Hei$ is collared. Indeed, since $B$ is a uniform domain by Lemma~\ref{lem-cap-tang} we are allowed to take $U=B$ and $h=Id$ in Definition~\ref{def-right-collared}.
\end{ex}

\begin{observ}\label{lem-unif-coll}
 Let $\Om\subset\Hei$ be a bounded uniform domain such that $\Om$ is an image of the unit ball $B(0,1)$ under a quasiconformal mapping $f$ with a continuous extension to a map $f:\partial U\to\partial B(0,1)$. Then, $\Om$ is collared.
\end{observ}

\begin{proof}
 In Definition~\ref{def-right-collared} let $U:=\Om$ and $h:=f$. Then (a) holds trivially, while $h(\partial U \cap \partial \Om)=h(\partial U)=\partial B(0,1)$ and thus (b) holds true as well.
\end{proof}

\begin{rem}
The following variant of the above definition of collardness at $x\in \partial \Om$ can be used as well. As in Definition~\ref{def-right-collared} we consider a a neighborhood $U \subset \Omega$, not necessarily uniform, such that $x \in \partial U$, but we require (b) to hold with respect to a global quasiconformal map $h: \Hei \to \Hei$ such that  $h(U)=B(0,1)$. The advantage of such a approach is that by Lemma~\ref{lem-cap-tang} together with Proposition 4.2 and Theorem 4.4 in \cite{ct}, we know that balls are uniform and the their image under global quasiconformal mappings are uniform.

In either the definition which we have settled on, or this alternative case, since $\Hei$ admits a large family of locally quasiconformal mappings, there is a large family of collared domains. Moreover, in \cite{Zolt}, Balogh constructs global quasiconformal maps on $\Hei$ which  are not Lipschitz and distort Hausdorff dimension. It is therefore possible that a collared domain can have a  complicated boundary, i.e., not rectifiable.
\end{rem}

In what follows we will appeal also to the following connectedness properties of the boundaries.
\begin{defn}\label{def-fin-con}
We say that $\Omega\subset \Hei$ is \emph{finitely connected at a point} $x \in \bd \Om$ if for every $r>0$ there exists a bounded open set $V$ in $\Hei$ containing $x$ such that $x\in V\subset B(x, r)$ and $V \cap \Omega$ has only finitely many components. If $\Omega$ is finitely connected at every boundary point, then we say it is finitely connected at the boundary.

In particular, if $V\cap \Om$ has exactly one component, then we say that $\Om$ is \emph{locally connected at} $x\in \bd \Om$
\end{defn}

\begin{defn}\label{def-collared-domain}
A domain $\Omega$ is said to be collared if every boundary point is locally quasiconformally collared.
\end{defn}

From now on by collared domain we will understand domains collared in the sense of Definition~\ref{def-right-collared}.

The following result extends discussion in Sections 6.3 and 6.4 in N\"akki~\cite{na3} and Theorem 17.10 in V\"ais\"al\"a~\cite{va1} to the setting of $\Hei$ and relates notions of collardness and boundary connectivity of the domain.

\begin{observ}\label{lem-coll-finit}
 Collared domains in $\Hei$ are locally connected at the boundary.
\end{observ}
We remark that using N\"akki's definition of collardness the above lemma stays that a collared domain is finitely connected at the boundary, see Theorem 6.4 and Corollary 6.6 in \cite{na3}.
\begin{proof}
 Let $\Om\subset \Hei$ be a collared domain and consider any $x\in \bd \Om$. Let $U$ be as in Definition~\ref{def-right-collared} with $x\in \bd U$. Let $B(x,r)\subset \Hei$ be as in Definition~\ref{def-fin-con}. Recall that $h(U)=B(0,1)$ is locally connected at the boundary, since $B(0,1)$ is uniform by Corollary 1 in Capogna--Garofalo~\cite{cg} and uniform domains are locally connected at the boundary, see Proposition 11.2 in \cite{abbs}. Therefore, we can choose an open connected set $V\subset \Hei$ with $h(x)\in \partial (V\cap B(0,1))$. Since $h$ is a homeomorphism, we obtain that $h^{-1}(V\cap B(0,1))$ is a connected subset of $U\cap B(x,r)$ and $x\in \partial (U\cap h^{-1}(V\cap B(0,1))$. Thus, $\Om$ is locally connected at $x$.
\end{proof}

The following observation will play a fundamental role in our studies.

\begin{observ}\label{obs-coll-spe}
 Let $\Om\subset \Hei$ be collared. Then for every boundary point $x\in\bd \Om$ there exists a singleton prime end $[E_k]$ such that $I[E_k]=\{x\}$.
\end{observ}

\begin{proof}
Let $x\in\bd \Om$ and $(U, h)$ be its collaring coordinate as in Definition~\ref{def-right-collared}. Let $x_0=h(x)\in \partial B(0,1)$ and define sets $E_k$ by
\begin{equation}\label{qc-coll-n}
E_k:=h^{-1}(\partial B(x_0,1/k) \cap B(0,1))
\end{equation}
for $k=1,2,\ldots$. It is an immediate observation that $F_k:=\partial B(x_0,1/k) \cap B(0,1)$ are cross-sets for all $k$, cf. Definition~\ref{def-cross-set}. Moreover, since by Lemma~\ref{lem-cap-tang} we have that $B$ is a uniform domain, then the definition of the sets $F_k$, and Lemma~\ref{sep-cond}, imply that $\{F_k\}_{k=1}^{\infty}$ is a prime chain in $B(0,1)$ as in Definition~\ref{def-pr-chain}. By construction $I[F_k]=\{x_0\}$. Definition of $E_k$ in ~\eqref{qc-coll-n} together with the uniformity of $U$ imply that every $E_k$ satisfies conditions (a) and (b) of Lemma~\ref{sep-cond},  and hence, by Theorem~\ref{herr}, we conclude that $\{E_k\}_{k=1}^{\infty}$ is a prime chain in $\Om$. Clearly $I[E_k]\subset \partial \Om$ and $x\in I[E_k]$. By Lemma~\ref{sep-cond} it holds that $\diam_{\Hei} E_k \to 0$ for $k\to \infty$ and, hence $I[E_k]$ is a singleton.
\end{proof}

 In what follows we will call such a prime chain a \emph{canonical prime chain} associated with $x\in \partial \Om$ and similarly the associated prime end will also be called a \emph{canonical prime end} associated with $x$.

\begin{thm}\label{thm-bij}
 If $\Omega$ is a collared domain, then the impression map $I:\bdyP \Om \to \bd \Om$ is a bijection.
\end{thm}

\begin{proof}
By Observation~\ref{obs-coll-spe}, the impression map is onto so we only need to show that it is injective. If the impression map is not injective then we can find distinct prime ends $[E_i]$ and $[F_j]$ such that $I[E_i]=I[F_j]=\{x\}$ for some $x\in \partial \Om$ or equivalently
\[
\bigcap_i \overline{D(E_i)} =  \bigcap_j \overline{D(F_j)} = \{x\}.
\]

Since $[E_i] \ne [F_j]$, we can assume that for each $j$ there exists $n_j \in \mathbb{N}$ such that $n_j \geq j$ and $E_i \nsubseteq D(F_j)$ for all $i \geq n_j$ (note that it in the last assertion it may be necessary to pass to a subsequence of the $\{E_i\}_{i=1}^{\infty}$ denoted again, for simplicity, by $\{E_i\}_{i=1}^{\infty}$). If in this case we have $E_i \cap D(F_j) = \emptyset$, then by choosing $i$ larger if necessary, we may assume  by (a) in Lemma~\ref{sep-cond} that $\overline{E}_i \cap \overline{D(F_j)} = \emptyset$. It then follows that  $\overline{D(E_i)} \cap \overline{D(F_j)} = \emptyset$ which contradicts $I[E_i]=I[F_j]=\{x\}$. Therefore, we must assume that  $E_i \cap D(F_j) \ne \emptyset$ for all $i \geq n_j$ which implies $E_i \cap \partial D(F_j) \ne \emptyset$ for all $i \geq n_j$. For each $i \geq n_j$ choose $x_i \in E_i \cap \partial D(F_j)$, it then follows that $x$ is a limit point of $\Omega \setminus D(F_j)$  since $x_i \to x$ and so $x \in \overline{F_j} \cap \partial \Omega$ which contradicts item (a) in the Definition~\ref{def-pr-chain} of prime chain, namely  $\Mod_{Q}(F_{j+1}, F_{j}, \Om)<\infty$. In particular, if $(U,h)$ is a collaring coordinate at $x$, then for $j$ sufficiently large we have $\overline F_j \subset U \cap \overline \Om$ and, since $\dist_{\Hei}(F_{j+1},F_j)=0$, Lemma \ref{herr} implies $\Mod_{Q}(F_{j+1}, F_{j}, U \cap \Om)=\infty$.  By the monotonicity of the $Q$-modulus it follows that $\Mod_{Q}(F_{j+1}, F_{j}, \Om)=\infty$ which contradicts the assumption that $\{F_j\}_{j=1}^{\infty}$ is a prime chain.
\end{proof}

In the following result we study the extension of a quasiconformal map to a homeomorphic transformation between the topological and the prime ends closures of a domain and the target domain, respectively. Results of this kind have long history and record, arising from Carath\'eodory's idea of prime ends. V\"ais\"al\"a~\cite[Section 3]{va2} proved it in the special case of a ball in $\R^n$, whereas N\"akki~\cite[Theorem 4.1]{na} studied the setting of collared domains in $\R^n$.

\begin{thm}\label{thm-key-res}
Let $\Om\subset \Hei$ be a collared domain and let $f:\Om\to\Om'$ be a quasiconformal map of $\Om$ onto a domain $\Om'\subset \Hei$. Then there exists a homeomorphic extension $\tilde f:\overline{\Om}\to \Om'\cup \partial_P \Omega' $ defined as follows 
\begin{equation*}
 \tilde f (x)=
 \begin{cases}
 f(x) & \hbox{if } x \in \Om \\
 [f(E_k(x))] & \hbox{if } x \in \bd \Om,
 \end{cases}
\end{equation*}
where $[E_k(x)]$ is the canonical prime end associated with $x \in \bd \Om$.
\end{thm}
\begin{proof} As a consequence of Theorem~\ref{thm-hom-1} and Theorem~\ref{thm-bij} above, we need only to check that the extension of the identity map $I_\Om:\Om \to \Om$ to a map $\tilde I_\Om: \Om \cup \partial \Om  \to \Om \cup \partial_P \Om$ where
\begin{equation*}
 \tilde I_\Om(x)=
 \begin{cases}
 x & \hbox{if } x \in \Om \\
 [E_k(x)] & \hbox{if } x \in \bd \Om,
 \end{cases}
\end{equation*}
is continuous and open. To this end we need only examine the behavior at the boundary.

Let $U \subset \Omega$ have the property that $\partial U \cap \partial \Om \ne \emptyset$ and that $U \cup (\partial U \cap \partial \Om)$ is relatively open, then for $U_P$ as in Definition~\ref{defn-topol} we have
\[
\tilde I_\Om\left(U \cup (\partial U \cap \partial \Om)\right)=U \cup U_P,
\]
since by collardness, every $[F_l] \in U_P$ satisfies  $[F_l]=[E_k(x)]$ for some $x \in   \partial U \cap \partial \Omega$. Hence $\tilde I_\Om$ is open.

Now let $U \cup U_P \subset \Om \cup \partial_P \Om$ where $U \cup (\partial U \cap \partial \Om)$ is relatively open. By collaredness, every $[F_l] \in U_P$ satisfies  $[F_l]=[E_k(x)]$ for some $x \in   \partial U \cap \partial \Omega$. Hence,
\[
U \cup U_P =\tilde I_\Om (U \cup (\partial U \cup \partial \Omega))
\]
and so $\tilde I_\Om$ is continuous.
\end{proof}

\section{Further properties of prime ends. Relations to the theory of prime ends on metric spaces}\label{sect4}

In this section we develop and discuss further properties of prime ends as defined in Section~\ref{sec-Nakkis-pe} in the setting of the Heisenberg group $\Hei$. Moreover, for domains in $\Hei$ we present relations between Nakki's prime ends and the theory of $\modq$-ends and $\modq$-prime ends as developed in~\cite{abbs}. We restrict our discussion here to $\Hei$ mainly for the sake of uniformity of the presentation. Most of the results in this section can be stated for the higher order Heisenberg groups and even for more general Carnot-Carath\'eodory groups under natural modifications of the statements below.

The following definition is due to N\"akki, see \cite{na2, na3}. N\"akki uses a term uniform domains, we call them mod-uniform domains, in order to distinguish from uniform domains as in Martio--Sarvas~\cite{MS}, see comments below.

\begin{defn}\label{def-uni-dom}
 We say that a domain $\Om \subset \Hei$ is \emph{mod-uniform} if  for every $t>0$ there is a $\epsilon>0$ such that if $\min\{\diam(E), \diam(F)\}\ge t$, then $\Mod_{4}(\Gamma(E, F, \Om))\geq \epsilon$ for any nondegenerate connected sets $E,F \subset \Om$.
\end{defn}

As observed by N\"akki, mod-uniform domains in $\overline{R^n}$ are finitely connected at the boundary, see Theorem 6.4 in \cite{na3}.  Moreover, a domain $\Om\subset \overline{R^n}$ finitely connected at the boundary is mod-uniform if and only if $\Om$ can be mapped quasiconformally onto a collared domain, see \cite[Section 6.5]{na3}. From the point of view of our discussion it is important that Theorem 6.4 in \cite{na3} easily extends to the $\Hei$ setting, and we omit the proof of this observation.

We further remark that one should not confuse Definition~\ref{def-uni-dom} with the uniform domains studied e.g. by~\cite{MS}, N\"akki--V\"ais\"al\"a~\cite{nv91} and V\"ais\"al\"a~\cite{vaisala88}, see Definition~\ref{def-unifDom} and Section~\ref{sect-loew-unif} for the importance of uniform domains in our studies. For instance, the latter uniform domains are necessarily locally connected at the boundary, see e.g. Proposition 11.2 in \cite{abbs}. In fact the following holds.

\begin{observ}\label{obs-uni-mod-uni}
 A uniform bounded domain $\Om\subset \Hei$ is mod-uniform.
\end{observ}

\begin{proof}
By Theorem~\ref{thm-uni-Loewn} we get that $\Om$ is $4$-Loewner. Let then $E, F\subset \Om$ be as in \eqref{def-Loewn1} and suppose that $\min\{\diam(E), \diam(F)\}\ge t$ for some $t>0$. Then
    $$\Delta(E, F)\leq \frac{1}{t}\distH(E,F)\leq \frac{1}{t}\diamH \Om$$
    and hence by Definition~\ref{def-Loewn}, it holds that
    $$\Mod_{4}(E, F, \Om)\geq \Psi_{\Om}(\Delta(E, F))>0.$$
   Moreover, since $\Om$ is bounded and $\Psi_{\Om}$ is a nonincreasing function, we in fact obtain that there exists a uniform lower bound
 \[
  \Psi_{\Om}(\Delta(E, F))\geq \Psi_{\Om}\left(\frac{1}{t}\diamH \Om\right):=\epsilon>0
 \]
 and the proof is completed.
\end{proof}

In the discussion following Definition~\ref{def-collared-domain} we noticed that every boundary point of a collared domain is an impression of the singleton prime end, the so-called canonical prime end. In fact, the following stronger result holds, cf. Observation~\ref{lem-coll-finit}.

\begin{lem}\label{obs-finite-pe}
 If $\Om\subset \Hei$ is a domain finitely connected at the boundary, then every $x\in \bd \Om$ is the impression of a prime end.
\end{lem}
The proof of this observation is based on the following topological result.
\begin{lem}[Lemma 10.5 in \cite{abbs}]\label{lem10-5}
 Assume that $\Om$ is finitely connected at $x_0\in\bd \Om$. Let $A_k\subsetneq\Om$ be such that:
\begin{enumerate}
\item $A_{k+1}\subset A_k$,
\item $x_0\in \overline{A_k}$,
\item $\distH(x_0,\Om\cap\bd A_k)>0$ for each $k=1,2,\ldots$.
\end{enumerate}
Furthermore, let $0<r_k<\distH(x_0,\Om\cap\bd A_k)$ be a sequence decreasing to zero.

Then for each $k=1,2,\ldots$ there is a component $G_{j_k}(r_k)$ of $B(x_0, r_k) \cap \Om$ intersecting $A_k$ for each $k=1,2,\ldots$, and such that $x_0\in \overline{G_{j_k}(r_k)}$ and $G_{j_k}(r_k) \subset A_k$.
\end{lem}

\begin{proof}[Proof of Lemma~\ref{obs-finite-pe}]
Following the notation of Lemma~\ref{lem10-5}, we let $x_0\in \bd \Om$ and set $A_k=\Om\setminus\{x\}$ for some $x\in \Om$ and all $k=1,\ldots$. First, we construct a nested sequence of connected sets
\[
 F^{x_0}_k=G_{j_k}(r_k)\subset B(x_0, r_k) \cap \Om
\]
 with $\diam_{\,\Hei}(F_k^{x_0})\to 0$ as $k\to \infty$. The idea of such construction is based on the proof of Lemma 10.6 in \cite{abbs} and, therefore, we present only a sketch of the reasoning.

  Let us consider the rooted tree with vertices $G_j(r_k)$, $j=1,2,\ldots,N(r_k)$, $k=1,2,\ldots$, where two vertices are connected by an edge provided that they are $G_j(r_k)$ and $G_i(r_{k+1})$ for some $i$, $j$ and $k$
with $G_i(r_{k+1}) \subset G_j(r_k)$. Denote by $\mathcal{P}$ the collection of all descending paths starting from the root and define a metric function measuring the distance between branches of the tree. Namely, let $t(p,q)=2^{-n}$, where $n$ is the level where paths $p$ and $q$ branch (or end), i.e. $n$ is the largest integer such that $p$ and $q$
have a common vertex $G_j(r_n)$. For each $k=1,2,\ldots$, we consider the subcollection $\mathcal{P}_k$ consisting of all paths $p\in\mathcal{P}$ for which there exists a component $G_j(r_{k})\subset A_k$ such that $p$ passes through the vertex $G_j(r_{k})$. By Lemma~\ref{lem10-5} all $\mathcal{P}_k$ are nonempty, $\mathcal{P}_{k+1}\subset \mathcal{P}_k$ for $k=1,\ldots$ and each $\mathcal{P}_k$ is complete in $t$. Since $\mathcal{P}$ is totally bounded in $t$, we get that all $\mathcal{P}_k$ are compact. In a consequence, we obtain an infinite path $q \in \bigcap_{k=1}^\infty \mathcal{P}_k$. The vertices through which it passes define the sequence of sets $\{F^{x_0}_k\}_{k=1}^{\infty}$ such that $F^{x_0}_k=G_{j_k}(r_k)$, $k=1,2,\ldots$. Moreover,
\[
 \diamH F^{x_0}_k \le \diamH (B(x_0, r_k) \cap \Om) \le 2r_k \to 0\quad\hbox{as}\quad k \to \infty.
\]

 Next, we use sets $F^{x_0}_k$ to define a prime chain $[E_k]$ with impression $I[E_k]=\{x_0\}$. Define
 \begin{equation}\label{chain-obs-finit-pe}
  E_k:=\left(\overline{F^{x_0}_{2k-1}}\cap \Om\right)\setminus \left(\overline{F^{x_0}_{2k}}\cap \Om\right),\quad k=1,\ldots.
 \end{equation}
 Then, $E_k$ are connected, relatively closed in $\Om$ for all $k$, also $\overline{E_k}\cap \bd \Om\not=\emptyset$ and $\Om\setminus E_k$ has exactly two components. By construction we get that $E_k$ separates $E_{k-1}$ and $E_{k+1}$. Furthermore, since $\dist_{\,\Hei}(E_k, E_{k+1})>0$ it holds that $\Mod_4(E_k, E_{k+1}, \Om)<\infty$.

  Finally, let $K\subset \Om$ be a continuum. Note that
  \[
   E_k\subset F^{x_0}_{2k-1}\subset B(x_0, r_k) \cap \Om \quad\hbox{ and }\quad \lim_{k \to \infty}\distH(B(x_0, r_k) \cap \Om, \{x_0\})=0.
  \]
   Therefore, $\lim_{k\to \infty}\Mod_4(E_k, K, \Om)=0$, as the family of curves passing through the fixed point has zero $p$-modulus for $1\le p\le Q=4$, cf. \eqref{obs-est1} below, for the similar argument. Thus, $[E_k]$ is a prime chain and defines a prime end.
\end{proof}

\begin{lem}\label{obs-finite-onto}
 Let $\Om_0\subset \Hei$ be a domain locally connected at the boundary and let $f$ be a quasiconformal mapping from $\Om_0$ onto a domain $\Om\subset \Hei$. Then, there exists a map $F:\bdyP \Om_0\to \bdyP \Om$, i.e. an image of a prime end in $\Om_0$ is a prime end in $\Om$.
\end{lem}
\begin{proof}
 Since $\Om$ is $1$-connected at the boundary, it is in particular finitely connected at the boundary and, hence, Lemmas~\ref{lem10-5} and \ref{obs-finite-pe} can be applied with sets
 \[
  F^{x_0}_{k}:=B(x_0, 1/k) \cap \Om_{0}\quad\hbox{for}\quad k=1,\ldots
 \]
   and any $x_0\in \bd \Om_0$. As in the proof of Lemma~\ref{obs-finite-pe} we construct a prime end $[E^x_k]$ in $\Om_0$ following \eqref{chain-obs-finit-pe}. Define $F:\bd \Om_0\to \bdyP \Om$ as follows:
 \[
  F(x)=[f(E^x_k)],\quad \hbox{for } x\in \bd \Om_0.
 \]
 The proof of the observation will be completed once we show that $\{f(E^x_k)\}_{k=1}^{\infty}$ defines a prime chain (end) in $\Om=f(\Om_0)$ with singleton impression $I[f(E^x_k)]:=\{y\}\subset \bd \Om$. Indeed, since $f$ is a homeomorphism, it holds that $f(E^x_k)$ are cross-sets in $\Om$ for all $k$. In particular, since for all $k$ cross-sets $E^x_k$ divide $\Om_0$ into exactly two domains, then so do $f(E^x_k)$ for all $k$. Similarly, by Topology
 we have that if $E_{k+1}$ separates $E_k$ and $E_{k+2}$, then the same holds for their images under homeomorphism $f$.
 Next, if $\dist_{\,\Hei}(E^x_k, E^x_{k+1})>0$, then by the injectivity of $f$ we have that $\dist_{\,\Hei}(f(E^x_k), f(E^x_{k+1}))>0$ for all $k$. Since $f$ is quasiconformal we infer from $\modq(E^x_k, E^x_{k+1}, \Om_0)<\infty$ that $\modq(f(E^x_k), f(E^x_{k+1}), \Om)<\infty$.

 Finally, since $[E^x_k]$ is a prime end in $\Om_0$, we have that for any continuum $K\subset \Om_0$
 \[
   0\leq \lim_{k\to \infty} \modq(f(E^x_k), f(K), \Om)\leq K\, \lim_{k\to \infty} \modq(E^x_k, K, \Om_0)=0
 \]
 by quasiconformality of $f$.  Note that every continuum $K'\subset \Om$ is an image under $f$ of some continuum in $K\subset\Om_{0}$, as we can set $K:=f^{-1}(K')$. This argument, together with Lemma~\ref{lem-imp-bd} imply that $I[fE_n]\subset \partial \Om$. The proof of Lemma~\ref{obs-finite-onto} is therefore completed.
 \end{proof}

\subsection{N\"akki's prime ends and prime ends on metric spaces}\label{sect-nakki-vs-pe}

In this section we compare a variant of N\"akki's theory of prime ends introduced in previous sections to a theory of prime ends developed for a general metric measure spaces in~\cite{abbs}. First, we recall building blocks of that theory.

Let $\Om\subset X$ be a domain in a complete metric measure space with a doubling measure, supporting the $(1,p)$-Poincar\'e inequality for $1<p<\infty$. For the importance of such assumptions we refer to \cite{abbs}. Here, we only note that those conditions hold for the Heisenberg groups $\mathbb{H}_n$ and more general Carnot--Carath\'eodory groups, see e.g. Section 11 in Haj\l asz--Koskela~\cite{HajKosk}.

\begin{defn}\label{def-accset}
 We say that a bounded connected set $E\subsetneq\Omega$ is an \emph{acceptable} set if $\overline{E}\cap \partial \Omega$ is nonempty.
\end{defn}

Since an acceptable set $E$ is bounded and connected it holds that $\overline{E}$ is compact and connected. Moreover, $E$ is infinite, as otherwise we would have $\overline{E}=E \subset \Om$. Therefore, $\overline{E}$ is a continuum.

\begin{defn}\label{def-chain}
A sequence $\{E_k\}_{k=1}^\infty$ of acceptable sets is a \emph{chain} if
\begin{enumerate}
\item \label{it-subset}
$E_{k+1}\subset E_k$ for all $k=1,2,\ldots$
\item \label{pos-dist}
$\distH(\Omega\cap\bd E_{k+1},\Omega\cap \bd E_k )>0$ for all $k=1,2,\ldots$
\item \label{impr}
The \emph{impression} $\bigcap_{k=1}^\infty \overline{E}_k \subset \bd\Om$.
\end{enumerate}
\end{defn}

We further comment that a variant of this definition can be considered as well with the Heisenberg distance in condition \eqref{pos-dist} substituted with the Mazurkiewicz distance, see Definition 2.3 in Estep--Shanmugalingam~\cite{es}.

\begin{defn}\label{def-my-ends}
Similarly to the setting of N\"akki's prime chains we define the \emph{division} of chains and say that
two chains are \emph{equivalent} if they divide each other. A collection of mutually equivalent chains is called an \emph{end} and denoted $[E_k]$, where $\{E_k\}_{k=1}^\infty$ is any of the chains in the equivalence class. An end $[E_k]$ is called a \emph{prime end} if any other end dividing it must be equivalent to it, i.e. if $[E_k]$ is not divisible by any other end.
\end{defn}

For further definitions and properties of prime ends as in Definition~\ref{def-my-ends} we refer to Sections 3-5 and 7 of \cite{abbs}. Among topics studied in \cite{abbs} are also notions of $\modp$-ends and $\modp$-prime ends for $1\leq p<\infty$, see Section 6 in \cite{abbs}. However, here we confine our discussion to the setting of $p=Q$ only with $Q=4$, the Ahlfors dimension of $\Hei$.
\begin{defn}
A chain $\{E_k\}_{k=1}^\infty$ is called a \emph{$\Mod_4$-chain} if
 \begin{equation}\label{def-modq-end-cond}
    \lim_{k\to \infty} \modq(E_k, K, \Omega)=0
 \end{equation}
for every compact set $K\subset \Omega$.
\end{defn}

In fact, Lemma A.11 in \cite{abbs} allows us to require \eqref{def-modq-end-cond} to hold only for some compact set $K_0$ with the Sobolev capacity $C_4(K_0)>0$.

\begin{defn}\label{def-my-pe}
An end $[E_k]$ is a \emph{$\Mod_4$-end} if it contains a $\Mod_4$-chain representing it. A $\Mod_4$-end $[E_k]$ is a \emph{$\Mod_4$-prime end} if the only $\Mod_4$-end dividing it is $[E_k]$ itself.
\end{defn}

\begin{rem}\label{rmk-interior}
Similarly to the prime chains studied in Section~\ref{sec-Nakkis-pe} it holds that the impression is either a point or a continuum, as $\{\overline{E}_k\}_{k=1}^\infty$ is a decreasing sequence of continua. Furthermore, Properties \ref{it-subset} and \ref{pos-dist} of Definition~\ref{def-chain} imply that $E_{k+1}\subset {\rm int} E_{k}$. In particular, ${\rm int} E_{k} \ne \emptyset$.
\end{rem}

\begin{thm}
 Let $\Om\subset \Hei$ be a collared domain. Then, a prime chain defines a $\modq$-chain as in Definition~\ref{def-my-pe} while a prime end defines a $\modq$-end according to Definition~\ref{def-my-pe}. Moreover, a prime end defines a $\modq$-prime end according to Definition~\ref{def-my-pe}.
\end{thm}
\begin{proof}
 Let $\{E_k\}_{k=1}^{\infty}$ be a prime chain in $\Om$. Recall, that by $D(E_k):=\Om\setminus E_k$ we denote the component of $\Om$ containing $E_{k+1}$ for $k=1,\ldots$. Then $D(E_k)$ are acceptable sets as in Definition~\ref{def-accset}, cf. Lemma~\ref{lem-imp-bd}. Moreover, $D(E_{k+1})\subset D(E_k)$ for all $k$. The definition of the prime chain together with Lemma~\ref{sep-cond} give us that since $\modq(E_k, E_{k+1}, \Om)<\infty$, then $\distH(E_k, E_{k+1})>0$ and hence $\distH(\Om\cap \bd D(E_k), \Om\cap \bd D(E_{k+1}))>0$. Again by Lemma~\ref{sep-cond} and the discussion following Definition~\ref{def-pr-chain} (see Lemma~\ref{lem-imp-bd}) it holds that the impression $\bigcap_{k=1}^{\infty}\overline{D(E_k)}\subset \bd \Om$. Hence, $\{D(E_k)\}_{k=1}^{\infty}$ defines a chain as in Definition~\ref{def-chain}. The $\modq$-condition for all continua assumed in Part (b) of N\"akki's Definition~\ref{def-pr-chain} implies that $\{D(E_k)\}_{k=1}^{\infty}$ is in fact a $\modq$-chain. Finally, since $[E_k]$ is a class of equivalent prime chains, we obtain that $[D(E_k)]$ is a $\modq$-end.

  Observation~\ref{lem-coll-finit} and Lemma~\ref{obs-finite-pe} imply that in a collared domain $\Om$ the impression of prime end $[E_k]$, and thus also the impression of $[D(E_k)]$, are singletons (contained in $\bd \Om$). Otherwise, $[E_k]$ is divisible by some prime end contradicting the definition of a prime end. Finally, recall that Proposition 7.1 in \cite{abbs} stays that a singleton end is a prime end (in the sense of Definition~\ref{def-my-ends}). In a consequence, $[D(E_k)]$ is a $\modq$-prime end (as in Definition~\ref{def-my-pe}).
\end{proof}

\section{Boundary behavior of quasiconformal mappings in $\Hei$}\label{sec-boundary}

 The main purpose of this section is to employ the theory of prime ends in the studies of the boundary behavior of quasiconformal mappings in the Heisenberg group $\Hei$.  Our results extend the corresponding ones proved in N\"akki~\cite[Section 7]{na}. We show counterparts of three results from the theory of conformal and quasiconformal mappings in $\R^n$:
 \begin{itemize}
 \item the Koebe theorem on existence of arcwise limits along end-cuts (Theorem~\ref{thm-Koebe}),
 \item the Lindel\"of theorem on relation between asymptotic values of a map and sets of principal points for prime ends (Theorem~\ref{thm-subs}),
 \item the Tsuji theorem on the Sobolev capacities of sets of arcwise limits (Theorem~\ref{tsuji-thm}).
 \end{itemize}

  These results require some definitions and auxiliary results, which we now present.

 Recall that if $\{E_k\}_{k=1}^{\infty}$ is a chain of cross-sets in $\Om$, then by $D(E_k)$ we denote the component of $\Om\setminus E_k$ containing $E_{k+1}$ (cf. Definition~\ref{def-chain}).

\begin{defn}\label{def-access-pt}
 A point $x\in\bd\Om$ is an \emph{accessible} boundary point if there exists a curve $\gamma:[0,1]\to \Hei$ such that $\gamma(1)=x$ and $\gamma([0,1))\subset\Omega$. We call $\ga$ an \emph{end-cut of $\Om$ from $x$}.

Moreover, if $[E_k]$ is a prime end and there is a curve $\ga$ as above such that for every $k$ there is $t_k\in(0, 1)$ with $\gamma([t_k,1))\subset D(E_k)$, then $x\in\bd\Om$ is \emph{accessible through $[E_k]$}.
\end{defn}

\begin{rem} $\phantom{AAA}$
\begin{enumerate}
\item Our definition of accessible boundary point differs from the one used in N\"akki~\cite{na}. Namely, in \cite[Section 7.1]{na} accessible point $x\in \bd \Om\subset \R^n$ is defined via \emph{closed Jordan arcs (loops)} lying entirely in the given domain except for possibly one endpoint $x$. The same definition would not work in
    $\Hei$ as there are no closed Jordan arcs among horizontal curves, see pg. 1846 in Balogh--Haj\l asz--Wildrick~\cite{bhw}.

\item Note that $x\in\bd\Om$ can be accessible through $[E_k]$ only if $x$ belongs to the impression of $[E_k]$.
\end{enumerate}
\end{rem}

The following result relates connectivity of the boundary of a domain to accessibility of points.
\begin{observ}\label{obs-acc-fc}
 Let $\Om\subset \Hei$ be a domain finitely connected at the boundary. Then every $x\in \bd \Om$ is accessible and accessible through some prime end $[E_n]$.
\end{observ}
\begin{proof}
 Lemma~\ref{obs-finite-pe} allows us to assign with every $x\in \bd \Om$ a prime end, denoted $[E_n]$, with $I[E_n]=\{x\}$. Moreover, $x$ is accessible through $[E_n]$ (cf. Definition~\ref{def-access-pt}). To see this choose $x_n\in D(E_n)$ for $n=1,2,\ldots$. Since both $x_n$ and $x_{n+1}$ belong to the pathconnected set $D(E_n)$, there exists a curve $\ga_n$ connecting $x_n$ to $x_{n+1}$. Let $\ga$ denote the union of all curves $\ga_n$, with $\ga([0,1))\subset \Om$ and $\ga(1)=x$. From the proof of Lemma~\ref{obs-finite-pe} we infer that $\lim_{n\to \infty}\diamH(E_n)=0$ and so $\ga$ is continuous at $1$.

 Hence, $x$ is accessible and accessible through $[E_k]$. Moreover, $\ga$ is an end-cut of $\Om$ from $x$.
\end{proof}

 Using Definition~\ref{def-access-pt} we may provide another method to associate with every accessible boundary point a prime end.  The following result will play a particular role in the studies of cluster sets of quasiconformal mappings (cf. Lemma 7.7 in \cite{abbs}). One can consider Observation~\ref{obs-acc-pe} as a complimentary result to Lemma~\ref{obs-finite-pe}.

 \begin{observ}\label{obs-acc-pe}
  Let $\Om\subset \Hei$ and $x\in \bd \Om$ be an accessible point. Let further $r_n$ for $n=1,2,\ldots$ be a strictly decreasing sequence converging to zero as $n\to\infty$. Then there exist a sequence $t_n$ for $n=1,2,\ldots$ with $0<t_n<1$ and a prime end $[E_n]$ such that:
\begin{enumerate}
\item $I[E_n]=\{x\}$,
\item $\ga([t_n,1))\subset E_n$,
\item $E_n$ is a component of $\Om\cap B(x, r_n)$ for all $n=1,2,\ldots$.
\end{enumerate}
In particular, $x$ is accessible through $[E_n]$. Moreover, $[E_n]$ is s singleton prime end.
 \end{observ}
\begin{proof}
 Let $\ga$ be an end-cut of $\Om$ from $x$ as in Definition~\ref{def-access-pt}. It is easy to notice that the continuity of $\ga$ implies existence of a sequence $t_n\in(0, 1)$ for $n=1,2,\ldots,$ with a property that
\[
\ga([t_n,1))\subset \Om \cap B(x,r_n).
\]
For $n=1,2,\ldots$ we define $D_n$ as the component of $\Om \cap B(x,r_n)$ containing
$\ga(t_n)$ and
\[
E_n:=(\overline{D_n}\setminus D_n)\cap \Om.
\]
We show that $\{E_n\}_{n=1}^\infty$ is a prime chain and, thus, gives rise to a prime end as in Definition~\ref{def-pr-end}.

By the definition, sets $E_n$ for all $n$ are relatively closed in $\Om$ and
\[
\overline{E_n}\cap \bd \Om=\overline{(\bd D_n\cap \Om)}\cap \bd \Om\not=\emptyset.
\]
Moreover, the choice of sets $D_n$ implies that every $\Om\setminus E_n$ consists of exactly two components whose
boundaries intersect $\bd \Om$. Hence, every $E_n$ is a cross-set as in Definition~\ref{def-cross-set}.

By construction $D_{n+1}\subset D_n\subset D_{n-1}$ and, since radii $r_n$ are strictly decreasing, we obtain that $E_n$ separates $E_{n-1}$ and $E_{n+1}$ for all $n=2,\ldots$. Hence, $\{E_n\}_{n=1}^{\infty}$ fulfills conditions of a chain, cf. Definition~\ref{def-chain-N}.

Since $E_n=\Om\cap\bd D_n \subset \bd B(x, r_n)$, it follows that for all $n=1,2,\ldots,$
\[
 \distH(E_n, E_{n+1}) \geq r_n - r_{n+1} > 0.
\]
In a consequence $\Mod_{4}(E_{n+1}, E_{n}, \Om)<\infty$ for all $n$. Finally, let $F\subset \Om$ be any continuum. Then for any $n$ we have that
\[
 \Mod_{4}(E_n, F, \Om)\leq  \Mod_{4}(\bd B(x, r_n), F, \Om)\to 0\quad \hbox{as}\quad n\to\infty.
\]
Hence, $\lim_{n\to \infty} \Mod_{4}(E_n, F, \Om)=0$ and, thus, conditions (a) and (b) of Definition~\ref{def-pr-chain}
are satisfied for $\{E_n\}_{n=1}^{\infty}$ and $[E_n]$ defines a prime end in $\Om$.

 Lemma~\ref{lem-imp-bd} implies that $I[E_n]\subset \bd \Om$. Since $E_n=(\overline{D_n}\setminus D_n)\cap \Om$ for $n=1,2,\ldots$, then $\overline{E_n}\subset \overline{\Om\cap B(x, r_n)}$ for all $n$. Hence,
 \[
  \diamH\,\overline{E_n}\leq \diamH\,\overline{\Om\cap B(x, r_n)}\to 0\quad\hbox{ for } n\to \infty
 \]
 by assumptions. Thus, $I[E_n]$ is a singleton prime end. In fact $I[E_n]=\{x\}$, as $x\in \overline{D_n}$ for all $n$ completing the proof of Observation~\ref{obs-acc-pe}.
\end{proof}

Recall the following notion of cluster sets.
\begin{defn}\label{def-cluster0}
Let $\Om\subset \Hei$ be a domain, $f:\Om\to \Hei$ be a mapping and $x \in \bd \Om$. We define \emph{the cluster set of $f$ at $x$} as follows:
\begin{equation}
C(f, x):=\bigcap_{U} \overline{f(U\cap \Om)}, \label{clust}
\end{equation}
where the intersection ranges over all neighborhoods of $x$ in $\Hei$.
\end{defn}

Cluster sets can be further generalized to capture the behavior of a mapping along a curve in a more subtle way.

\begin{defn}\label{def-cluster}
 Let $\Om\subset \Hei$ be a domain, $f:\Om\to \Hei$ be a mapping and $x\in \bd \Om$.  We say that a sequence of points $\{x_n\}_{n=1}^{\infty}$ in $\Om$ \emph{converges along an end-cut $\ga$ at $x$} if there exists a sequence $\{t_n\}_{n=1}^{\infty}$ with $0<t_n<1$ and $\lim_{n\to\infty} t_n=1$ such that $x_n=\ga(t_n)$ and
 \[
 \lim_{n\to \infty}\dH(x_n, x)=0.
 \]
 We say that a point $x'\in \Hei$ belongs to \emph{the cluster set of $f$ at $x$ along an end-cut $\ga$ from $x$}, denoted by $C_{\ga}(f, x)$, if there exists a sequence of points $\{x_n\}_{n=1}^{\infty}$ converging along an end-cut $\ga$ at $x$, such that
 \begin{equation}
 \lim_{n\to\infty} \dH(f(x_n), x')=0.
 \end{equation}

 If $C_{\ga}(f, x)=\{y\}$, then $y$ is called an \emph{arcwise limit(asymptotic value)} of $f$ at $x$.
\end{defn}
In other words, $y$ is an asymptotic value of $f$ at $x\in\Om$, if there exists a curve $\ga:[0, 1)\to\Om$ such that
$\ga(t)\to x$ and $f(\ga(t))\to y$ for $t\to 1$.
\medskip

{\bf The Koebe theorem}
\medskip

In 1915 Koebe~\cite{ko} proved that a conformal mapping between a simply-connected planar domain $\Om$ onto the unit disc has arcwise limits along all end-cuts of $\Om$. The following result extends Koebe's theorem and Theorem 7.2 in N\"akki~\cite{na} to the setting of quasiconformal mappings in $\Hei$. Moreover, we study more general end-cuts than in \cite{na}.

\begin{thm}[The Koebe theorem in $\Hei$]\label{thm-Koebe}
 Let $f:\Om\to \Om_0$ be a quasiconformal map between a domain $\Om\subset \Hei$ finitely connected at the boundary and a mod-uniform domain $\Om_0\subset \Hei$. Then $f$ has arcwise limits allong all end-cuts of $\Om$.
\end{thm}

\begin{proof}
 We follow the idea of the proof of \cite[Theorem 7.2]{na}. Since $\Om$ is finitely connected at every boundary point $x\in \bd \Om$, then Observation~\ref{obs-acc-fc} implies that all $x\in \bd \Om$ are accessible. Let $\ga$ be an end-cut in $\Om$ from $x\in \Om$. Let $K\subset \Om$ be a continuum and let $U_k$ be neighborhoods of $x$ such that
 \[
  \bigcap_{k=1}^{\infty} U_k=\{x\}\quad \hbox{and}\quad \ga_k:=U_k\cap \Om\cap \ga
 \]
 are connected sets for all $k\geq k_0$ and some $k_0$. Since $\diamH U_k\to 0$ for $k\to \infty$ and $\ga_k\subset U_k$ we have that
 \[
 \lim_{k\to \infty} \Mod_4(K, \ga_k, \Om)\leq \lim_{k\to \infty} \Mod_4(K, U_k, \Om)=0.
  \]
  In order to see that latter let us assume that $R>0$ is sufficiently small and such that $B(x, R)\subset \Hei\setminus K$. By the decay property of diameters for sets $U_k$ for $k\to \infty$ we have that, passing to a subsequence if necessary, $U_k\subset B(x, 1/k)\cap \Om$ for all $k\geq k_0$. Since $\Hei$ is a path-connected metric measure space with doubling measure, then Theorem 3.1 in Adamowicz--Shanmugalingam~\cite{ash} together with the fact that the modulus of curve families is an outer measure imply that
 \begin{align}
  \Mod_4(U_k, K, \Om)&\leq \Mod_4(B(x, 1/k)\cap \Om, K, \Om) \nonumber \\
  & \leq \Mod_4(B(x, 1/k)\cap \Om, K, B(x, R)) \nonumber  \\
  &\leq  \Mod_4(\overline{B(x, 1/k)}, \Hei\setminus B(x, R), \Hei) \leq  \Mod_4(\overline{B(x, 1/k)}, \Hei\setminus B(x, R), B(x, R)) \nonumber \\
  &\leq C(R) \left(\log \frac{R}{1/k}\right)^{-3}\to 0,\quad \hbox{for } k\to \infty. \label{obs-est1}
 \end{align}
 In the fourth inequality above we also used the fact that the family of curves $\Gamma_2(\overline{B(x, 1/k)}, \Hei\setminus B(x, R), \Hei)$ is minorized by $\Gamma_1(\overline{B(x, 1/k)}, \Hei\setminus B(x, R), B(x, R))$, see Section~\ref{sec-curves}.

  The quasiconformality of $f$ implies that
  \[
   \lim_{k\to \infty} \Mod_4(f(K), f(U_k), f(\Om))=0.
  \]
  Since $\Om_0$ is a mod-uniform domain, Definition~\ref{def-uni-dom} gives us that $\lim_{k\to\infty} \diam(f(\ga_k))=0$ and thus the cluster set $C_{\ga}(f, x)$ is a singleton meaning that $f$ has an arcwise limit along $\ga$.
 \end{proof}
\medskip

{\bf The Lindel\"of theorem}
\medskip

A bounded analytic function of the unit disc having a limit $y_0$ along an end-cut  at a boundary point $x_0$ has angular limit $y_0$ according to the classical theorem of Lindel\"of. By angular limit we mean that the limit is $y_0$ along any ``angular'' end-cut at $x_0$, that is an end-cut contained in some fixed cone in the unit disc with apex at $x_0$.

In \cite[Theorem 6]{gehr1}, Gehring proved a Lindel\"of type theorem for quasiconformal mappings on balls in  $\mathbb{R}^3$ which N\"akki generalized to $\mathbb{R}^n$ in \cite[Theorem 7.4]{na}. In this context the theorem is stated in terms of angular end-cuts and principal points.

 \begin{defn}\label{princ-pt}
  Let $[E_k]$ be a prime end in a domain $\Om\subset \Hei$ and $x\in I[E_k]$. We say that $x$ is a \emph{principal point} relative to the prime end $[E_k]$, if every neighborhood of $x$ contains a cross-set of a chain in $[E_k]$, i.e., $x$ is a limit of a convergent chain in $[E_k]$. The set of principal points of a prime end $[E_k]$ is denoted by $\Pi(E_k)$. A point in $I[E_k]$ which is not principal is called a \emph{subsidiary point}.
 \end{defn}

For the main ideas and definitions of principal (and subsidiary points, see below) we refer to Collingwood-Lohwater~\cite[Chapter 9.7]{CL} and N\"akki~\cite[Section 7]{na}. The importance of such notions in the classification of prime ends in $\R^n$ is described e.g. in \cite[Section 8]{na}. See also Carmona--Pommerenke~\cite{cp1, cp2} for results regarding the theory of continua and principal points.

The geometry of $\Hei$ imposes a number of obstacles in proving a Lindel\"of type theorem. Firstly, the Euclidean proof relies on transforming the setting to the upper half space by a stereographic projection which sends $x_0$ to the origin and then using the fact that a cone in $\mathbb{R}^n$ with apex at $0$ is invariant under dilation. Although we have stereographic projections, we do not have the luxury of choosing the destination of $x_0$. Secondly, we do have a notion of a cone in a domain $\Omega$  with apex $x_0 \in \partial \Omega$ and width $\alpha$,  i.e., the set $\{ x \in \Omega : d_{\Hei}(x, x_0) \leq \alpha d_{\Hei}(x, \partial \Omega\}$, however these sets are not invariant under dilations centered at $x_0$, that is maps of the form $g_s=\tau_{x_0} \circ \delta_s \circ \tau_{x_0}^{-1}$. In particular, the definition of angular end-cut in $\Hei$ via cones lacks convenient properties. A more natural and geometrically convenient notion is the following definition of what we call a contractible end-cut.

\begin{defn}\label{angular-end-cut}
  Let $\ga$ be an end-cut in the ball $B(0,1)\subset \Hei$ from a point $x_0 \in \bd B$. We say that $\ga$ is \emph{contractible} if for each interval $[t,1) \subset [0,1)$ we have $\ga([t_0,1)) \not\subset g_s(B(0,1))$ for $s$ sufficiently small.

  For a general ball we use translations and dilations to normalize and apply the definition we have given for $B(0,1)$.
  \end{defn}

We now explain why such a definition is geometrically suitable. We consider the normalized situation where $B=B(0,1)$ and $x_0 \in \partial B(0,1)$. For a given sequence $r_k \to 0$ we define a sequence of contractions $g_k$, all with fixed point $x_0$, by setting $g_k=\tau_{x_0} \circ \delta_{r_k} \circ \tau_{x_0}^{-1}$. If $x \in \overline {B(0,1)}$ and $x\ne x_0$, then
$$d(0, g_k^{-1}(x))=\frac{1}{r_k} d(g_k(0),x). $$
By direct calculation we have that $$ d(g_k(0),x) =\left (\sum_{k=0}^4 A_k(x_0,x) r_k^k \right )^{\frac14},$$
where $A_0(x_0,x)=||x_0||=1$ and the remaining $A_k$ are polynomials of homogeneous degree $4$. Therefore, if $n$ is sufficiently large we have
$$d(0, g_k^{-1}(x))= \left (\sum_{k=0}^4 A_k(x_0,x) r_k^{k-4}  \right )^{\frac14}   $$
and so if $\alpha_k(x_0):=\min \{ A_k(x_0,x) : x \in \overline {B(0,1)}\}$, then
$$
d(0, g_k^{-1}(x)) \geq  \frac{1}{r_k} \left (1 + \alpha_1(x_0)r_k  + \alpha_2(x_0) r_k^2 + \alpha_3(x_0)r_k^3 + \alpha_4(x_0)r_k^4    \right )^{\frac14}
$$
for all $x \in \overline {B(0,1)} \setminus \{x_0\}$. Hence, on $\overline {B(0,1)}  \setminus B(x_0,\epsilon)$  we have $d(0, g_k^{-1}(x)) \geq 1$ for $n$ sufficiently large and
$$ g_k^{-1} \left ( \overline {B(0,1)}  \setminus B(x_0,\epsilon) \right ) \subset \Hei \setminus B(0,1).  $$

Hence the convenience is encapsulated in the following conclusion: If $\epsilon >0$ is given, then for $k$ sufficiently large, we have
\begin{align}
D_k:=g_k(B(0,1) ) \subseteq B(0,1) \cap B(x_0,\epsilon). \label{contraction-lemma}
\end{align}

The reason we choose $\ga([t_0,1)) \not\subset g_s(B(0,1)$ is so that geodesic rays are contractible. Indeed, let  $\phi(s, x_0)$ denote the radial geodesic joining the origin to $x_0$, see \eqref{ray}, then
\begin{align*}
||g_{k}^{-1}(\phi(s, x_0))|| &=||\tau_{x_0} \circ \delta_{1/r_k} \circ \tau_{x_0}^{-1} (\phi(s, x_0))||\\
&= \frac{1}{r_k} ||\phi(s, -x_0)||\\
&= \frac{s}{r_k}||x_0||\\
&= \frac{s}{r_k}.
\end{align*}
Hence we can choose $r_k$ sufficiently small so that $||g_{r_k}^{-1}(\phi(s, x_0))|| >1$ for all $s \in (t,1] \subset (0,1)$ and so the radial geodesic avoids $D_k$ for all $r_k<1$.

  Theorem~\ref{thm-subs} below is an analog of the Lindel\"of theorem and corresponds to Theorem 6 in Gehring~\cite{gehr1} and Theorem 7.4 in N\"akki~\cite{na}. See also Vuorinen~\cite{vuo} for related studies in the context of angular limits for quasiregular mappings in $\R^n$ and N\"akki~\cite{na2010} for further relations between angular end-cuts and various types of cluster sets.

   In the proof of Theorem~\ref{thm-subs} we will need the following auxiliary result. Recall that if $x$ is any boundary point of a collared domain, then we can associate with $x$ a so-called canonical prime end, cf. the discussion following \eqref{qc-coll-n}.

   \begin{observ}\label{obs-aux-subs}
    Let $\Om\subset \Hei$ be a collared domain and $f$ be a quasiconformal embedding of $\Om$ into $\Hei$. For any $x\in \bd \Om$ and a canonical prime end $[E_k(x)]$ with impression $x$, it follows that $[f(E_k(x))]$ is a prime end in $f(\Om)$.
   \end{observ}
   \begin{proof}
     It is an immediate consequence of the topological properties of the homeomorphism $f$, that $E_k=f(E_k(x))$ is a cross-set for $k=1,2,\ldots$ as in Definition~\ref{def-cross-set}. In order to show that ${E_k}$ is a prime chain, and thus a prime end in $f(\Om)$, we need to verify conditions (a) and (b) of Definition~\ref{def-pr-chain}. By quasiconformality of $f$ it holds that
    \[
      \Mod_{4}(f(E_{k+1}), f(E_{k}), f(\Om))\leq K \Mod_{4}(E_{k+1}, E_{k}, \Om)<\infty.
    \]
    Similarly, if $F \subset \Om$ is any continuum we have that
    \[
    \lim_{k\to \infty} \Mod_{4}(f(E_k), f(F), f(\Om))\leq K \lim_{k\to \infty} \Mod_{4}(E_k, F, \Om)=0.
    \]
    Hence, $[E_k]$ satisfies Definition~\ref{def-pr-chain}.
   \end{proof}

  As observed in \eqref{qc-coll-n}, to every point in a collared domain we can associate a canonical prime end denoted $[E_k(x)]$ such that $I[E_k(x)]=\{x\}$. In particular this holds in a ball $B$. By Observation~\ref{obs-aux-subs}, if $f$ is quasiconformal, then a chain $\{f(E_k(x)\}$ is a prime end.

  \begin{thm}[The Lindel\"of theorem in $\Hei$]\label{thm-subs}
   Let $f$ be a bounded quasiconformal mapping of a ball $B\subset \Hei$ onto a domain $\Om_0 \subset \Hei$ with the property that $\diam_{\Hei}(f(\partial B(x,r) \cap B))\to 0$ as $r \to 0$. Then for a.e.  $x_0 \in \partial B$ it holds that for every contractible end-cut $\ga$ of $B$ from $x_0$ we have
   \[
    C_{\ga}(f, x_0)=\Pi(f(E_k(x_0)).
   \]
Note that a.e. can be taken relative to the Radon measure on $S(0,1)$ discussed in the subsection on polar coordinates in appendix A.
   \end{thm}

\begin{proof}[Proof of Theorem~\ref{thm-subs}]

Let us begin by noting that there is no loss of generality if we assume $B=B(0,1)$ and $f(0)=0$.

First we show that for a.e. $x_0 \in S(0,1)$ we have $\Pi(f(E_k(x_0)) \subset C_{\ga}(f, x_0)$. By definition, every neighborhood of $y$ contains $f(E_{k}(x_0))$ for some $k$, and the image of the radial geodesic satisfies $f(\phi(s_k,x_0)) \in f(E_{k}(x_0))$ for some $s_k$. Since quasiconformal maps are ACL, it follows that $f(\phi(s,x_0))$ is a horizontal curve for almost all $x_0$, moreover $f(\phi(s_k,x_0)) \to y$ as $k \to \infty$.

Now we show that $C_{\ga}(f, x_0) \subset \Pi(f(E_k(x_0))$. Let $y \in C_{\ga}(f, x_0)$, i.e., there exists a sequence $t_n \to 0$ such that $s_n=d_{\Hei}(\gamma(1-t_n), x_0) \to 0$ as $n$ tends to $\infty$. Let $E_{s_n}(x_0) = S(x_0,s_n)\cap B$, then $[f(E_{s_n}(x_0))]$ is a prime end in $\Omega_0$ since $f$ is quasiconformal. We want to show that $y$ is a principal point of $[f(E_{s_n}(x_0))]$, i.e., every neighborhood of $y$ contains $f(E_{s_n}(x_0))$ for some $n$. More precisely, we will show that for every $\epsilon >0$, we have $f(E_{s_n}(x_0)) \subset B(y, \epsilon) \cap f(B)$ for $n$ sufficiently large.

Define a sequence of mappings
\[
 f_n=f \circ g_n,\hbox{ for } n=1,2,\ldots
\]
and $g_n$ as in \eqref{contraction-lemma}. Since $f$ is bounded and $f(0)=0$, it follows from  \eqref{contraction-lemma}  that $f_n$ avoids the values $0$ and $\infty$ if we consider $f$ as a mapping $f:B(0,1) \to  \hat {\mathbb{H}} $. By \cite{kr1}, p.321, $f_n|_{B(0,1)}$ corresponds conformally with a sequence of $K$-qc functions $\hat f_n: \hat B(0,1) \to S^3$ where $S^3=\partial \hat B(0,1) \subset \mathbb{C}^2$ and $\hat B(0,1) \subset S^3$.   Here the ball $\hat B(0,1)$ and the conformallity are with respect to the spherical metric $$ d^S(u,w)^2=2|1-(u,w)|=||u-w|^2-2 i {\rm Im}\, (u,w)|$$ where $(u,w)=u_1\bar w_1+u_2 \bar w_2$. Moreover, $\hat f_n$ avoids the values in $S^3$ corresponding with $0$ and $\infty$ by a fixed positive distance for all sufficiently large $n$. By \cite{kr2} Theorem F, the sequence  $\hat f_n$ is normal. Hence there exists a subsequence $f_{n_j}$ which converges uniformly on compact subsets of $B(0,1)$ to a $K$-qc mapping $h$ or a constant.

If $j$ is sufficiently large then we have  $f_{n_j} (B(0,1)) \subset f( B(0,1) \cap B(x_0,\epsilon))$ and so by accelerating the subsequence $f_{n_j}$ we can assume
\begin{equation}
f_{n_j} (B(0,1))  \subset f( B(0,1) \cap  B(x_0, s_{n_j}) ).\label{fndinclude}
\end{equation}
This implies that
$$ f_{n_j} (B(0,1))  \subset D(f(E_{s_{n_j}} (x_0))).$$
As a consequence we then have that $$ \bigcap_k \overline {f_{n_j} (B(0,1))} \quad  \subset \quad \bigcap_k \overline{D(f(E_{s_{n_j}} (x_0)))}\quad  = \quad I[f(E_{r_{s_{n_j}} } (x_0))].$$

Let $U \subset B(0,1)$ be compact, then $f_{n_j} \to h$ or a constant uniformly on $U$. If $h$ is the nonconstant limit, then $h(U) \subset \partial f(B(0,1))$ which contradicts the fact that $h$ is homeomorphism. Hence $f_{n_j}$ converges uniformly to a constant value on $U$. Choose $U$ so that $ E_{1}(x_0) \cap  U \ne \emptyset$ then it follows that  $f_{n_j}(E_{1}(x_0)\cap U) \subset f(E_{s_{n_j} } (x_0))$
and $$\lim_{j\to \infty} f_{n_j}(E_{1}(x_0)\cap U)=y(x_0,U) \in  I[f(E_{s_{n_j} } (x_0))].$$   It follows that $f(B) \cap B(y(x_0,U),\varepsilon)$ contains  $f_{n_j}(E_{1}(x_0)\cap U)$ for $k$ sufficiently large and so  $$f(B) \cap B(y(x_0,U),\varepsilon) \cap f(E_{s_{n_j} } (x_0)) \ne \emptyset.$$ By assumption,  $\diam_{\Hei}(f(E_{s_{n_j} } (x_0)) \to 0$ and so it follows that $y(x_0,U)$ is independent of $U$ and is in fact equal to $y$. Moreover the sets $f(E_{s_{n_j} } (x_0))$ satisfy the requirements that qualify $y$ as a principal point of  $I[f(E_{s_{n_j} } (x_0))]$.  Since $[E_k(x_0)]=[E_{s_{n_j}}(x_0)]$ it follows that  $[f(E_k(x_0))]=[f(E_{s_{n_j}}(x_0))]$ hence we have $C_{\ga}(f, x_0) \subset \Pi(f(E_k(x_0))$.
\end{proof}

\begin{rem} The assumption that $\diam_{\Hei}(f(\partial B(x,r) \cap B))\to 0$ as $r \to 0$ is not needed in the Euclidean case since it can be shown that  $\diam(f(\partial B(x,r) \cap B))\to 0$ in the spherical metric. See  \cite[Theorem 7.4]{na}, \cite[Theorem 6]{gehr1}. The proof does not carry over trivially to the setting of the Heisenberg group and as yet we do not know if such a result holds.

The fact that we prove the theorem for a.e. $x_0$ is only required so that that we can employ the images $f(\phi(s,x_0))$ as horizontal curves. If we could show that the points $f(\phi(s_k,x_0))$ can be joined by horizontal curves in $\Omega_0$ then the result could be strengthened to all $x_0$.
\end{rem}
\medskip

{\bf The Tsuji theorem}
\medskip

   Our next goal is to show the quasiconformal counterpart of the Tsuji theorem in $\Hei$. A theorem due to F. and M. Riesz states that if a planar bounded analytic function in the unit disk $B^2$ has the same radial limit in a set of positive Lebesgue measure in $\bd B^2$, then the function is constant, see e.g. Theorem 2.5 in Collingwood--Lohwater~\cite[Chapter 2]{CL}. The celebrated example due to L. Carleson~\cite{carl} shows that the weaker version of that result, with radial limits existing in a boundary set of a positive logarithmic capacity, is false. However, Tsuji proved that the set of boundary points with the same radial limit $\alpha$ is of zero logarithmic capacity, provided that $\alpha$ is an ordinary point of the analytic function, see Theorem 5 in Tsuji~\cite{tsu} for details and Villamor~\cite{vill} for further studies of Tsuji's result.
   In \cite[Theorem 6]{tsu} Tsuji proved also the following result: consider a conformal map between $B^2$ and a planar simply-connected domain $\Om$ with the set $\mathcal{A}$ of accessible points in $\partial \Om$ of zero capacity.
   Then, the set of points in $\bd B^2$ corresponding to $\mathcal{A}$ has zero capacity as well.
    This result was extended to the setting of quasiconformal mappings in $\overline{R^n}$ by
     N\"akki~\cite[Theorem 7.12]{na}. The following theorem generalizes N\"akki's result in $\Hei$.

   In the statement of Theorem~\ref{tsuji-thm} below we use the notion of arcwise limit, cf. Definition~\ref{def-cluster}.
   Furthermore, the Tsuji theorem in $\Hei$ relies on three notions which we now define: an arcwise extension of a quasiconformal map and the Sobolev and the condenser capacities.

 Let $\Om\subset \Hei$ be a collared domain and $f$ be a quasiconformal homeomorphism of $\Om$. Denote by $A_f\subset \bd f(\Om)$ a set of arcwise limits of $f$ at $\bd \Om$. Similarly, denote by $A$ a set of points in $\bd \Om$ where $f$ has an arcwise limit.

   A map $F:\Om\cup A\to f(\Om)\cup A_f$ is called an \emph{arcwise extension} of $f$ and is defined as follows:
    \[
     F(x)=
     \begin{cases}
      &f(x),\quad x\in \Om\\
      &C_{\ga}(f, x)\subset \bd f(\Om),\quad x\in A.
     \end{cases}
    \]

    A map $F$ is well-defined, i.e., for every $x\in A$ it holds that $C_{\ga}(f, x)=\{y\}$ for some $y\in A_f$ along any end-cut $\ga$ at $x$. Indeed, by Observation~\ref{lem-coll-finit} and Observation~\ref{obs-acc-fc} we get that every $x\in \bd \Om$ is accessible, accessible through some prime end $[E^x_k]$ and there exists an end-cut of $\Om$ from $x$. Theorem~\ref{thm-key-res} gives us that cluster set $C(f, x)=I[f(E^x_k)]$. This observation, combined with an immediate observation that $C_{\ga}(f, x)\subset C(f, x)$ for any end-cut $\ga$ at $x$, gives us that $C_{\ga}(f, x)\subset I[f(E^x_k)]$ for every $x$ and its every end-cut. By Lemma~\ref{obs-finite-pe} we know that $I[E^x_k]=\{x\}$. The argument will be completed once we show that $I[f(E^x_k)]$ is a singleton. By the definition of $C_\ga(f, x)$ for $x\in A$ we have a sequence of points $\{f(x_n)\}_{n=1}^{\infty}$ and, thus, by joining these points we may build a curve, denoted $\ga'$. By constructions in Observations~\ref{obs-acc-fc} and \ref{obs-acc-pe} we obtain a prime end, denoted $[E_{\ga'}]$, whose impression satisfies
    \begin{equation}\label{aux-Tsuji}
    C_{\ga'}(f, x)\subset I[E_{\ga'}]\subset I[f(E^x_k)].
    \end{equation}
     If a different end-cut $\ga_1$ at $x$ provides $C_{\ga_1}(f, x)\not=C_{\ga'}(f, x)$, then we obtain a different prime end, denoted $[E_{\ga_1}]$ satisfying inclusions similar to \eqref{aux-Tsuji}. Therefore, both $[E_{\ga'}]$ and $[E_{\ga_1}]$ divide $[f(E^x_k)]$ contradicting that $[f(E^x_k)]$ is a prime end. Hence $C_{\ga}(f, x)$ is the same singleton set for all end-cuts $\ga$ at $x$, and so is $I[f(E^x_k)]$.

     Thus, $C_{\ga}(f, x)$ is a single point (asymptotic value) independent of the choice of end-cut $\ga$ and $F$ is well-defined for every $x\in \bd A$. 

 Finally, we recall the notions of the Sobolev capacity and the condenser capacity specialized to the case of the Heisenberg group $\Hei$.

The \emph{Sobolev $4$-capacity} of a set $E\subset \Hei$ is defined as follows:
\begin{equation*}
  C_p (E) :=\inf \|u\|_{N^{1, 4}(\Hei)}^4,
\end{equation*}
where the infimum is taken over all Newtonian functions $u\in N^{1, 4}(\Hei)$ such that $u \geq 1$ on $E$ (see e.g. \cite{hkstB} for definitions and properties of Newtonian spaces).

 The following result relates the modulus of curve families to the condenser capacity, see Definition 1.4 and Remark 1.9 in Vuorinen~\cite{vuo1}, also Ziemer~\cite{ziem69}.

\begin{lem}[cf. Lemma A.1 in \cite{abbs}]\label{mod-cap}
For any choice of a compact subset of a ball $K\subset B_{R}$ we have that
\begin{equation}\label{eq-capp=modp}
  \Mod_4 (K, \bd B_{R}, B_{R}) =  \Mod_4 (K, \bd B_{2R}, \Hei)= {\rm cap}_4(K, \bd B_{R}, B_R)\geq {\rm Cap}_4(K, B_{R}),
\end{equation}
where ${\rm Cap}_4(K, B_{R})$ denotes the $4$-capacity of the condenser $(K, B_{R})$ and is defined by
 \begin{equation}\label{eq-capp-minimizer}
 {\rm Cap}_4(K, B_{R}):=\inf_u \int_\Om g_u^4\,d\lambda,
 \end{equation}
 where $g_u$ stands for a $4$-weak upper gradient of $u$ and the infimum is taken over all $u\in N_{0}^{1,4}(B_{R})$ satisfying $u\geq 1$ on $K$.
\end{lem}

 Let us just comment that the first equality in \eqref{eq-capp=modp} follows from the properties of the $p$-modulus: the first curve family is contained in the second, but the second is minorized by the first one. The equality between the $4$-modulus and the $4$-capacity is a consequence of \cite[Lemma A.1]{abbs}.

 Recall Definition~\ref{def-uni-dom} of mod-uniform domains.

   \begin{thm}[The Tsuji theorem in $\Hei$]\label{tsuji-thm}
    Let $f$ be a quasiconformal mapping of a ball $B=B_R\subset \Hei$ of radius $R$ such that $f(B)$ is a mod-uniform domain and let $F:\overline{B}\to f(B)\cup A_f$ be an arcwise extension of $f$.

    If $A_f$ is compact and $C_4(A_f)=0$, then $C_4(F^{-1}(A_f))=0$.
   \end{thm}

  \begin{rem}
   Assumptions of Theorem~\ref{tsuji-thm} simplify if $f$ is a global quasiconformal map $f:\Hei\to\Hei$. Then $B$ is a uniform domain, by Lemma~\ref{lem-cap-tang}, and moreover, $f(B)$ is a uniform domain, by Proposition 4.2 and Theorem 4.4 in \cite{ct}. Thus, $f(B)$ is mod-uniform by Observation~\ref{obs-uni-mod-uni}.
  \end{rem}

   \begin{proof}
    The main idea of the proof is similar to the one of Theorem 7.12 in \cite{na}. However, we need to adjust several tools and auxiliary results to the Heisenberg setting. Moreover, we use techniques developed in recent years.

    Let $E\subset f(B)$ be a closed set. Denote by $\Gamma(E, A_f, f(B))$ the family of (horizontal) curves $\ga$ in $f(B)$ with one endpoint in $E$, the other in $A_f$, and $\ga\setminus (E\cup A_f)\subset f(B)$.

    By Proposition 1.48 in Bj\"orn--Bj\"orn~\cite{bb} applied to the setting of $\Hei$ we get that if
    $C_4(A_f)=0$, then $\Mod_4(\Gamma_{A_f})=0$. Here, $\Gamma_{A_f}$ denotes the family of all (horizontal) curves in $\Hei$ passing through $A_f$. Since $\Gamma(E, A_f, f(B))\subset \Gamma_{A_f}$ it holds that
    \[
     \Mod_4(E, A_f, f(B))\leq \Mod_4(\Gamma_{A_f})=0.
    \]
    We set $\Delta':=\fv \Gamma(E, A_f, f(B))$ and use the quasiconformality of $f$ to conclude that $\Mod_4(\Delta')=0$. Since $B$ is collared and $f(B)$ is mod-uniform, then by the Koebe theorem, Theorem~\ref{thm-Koebe} we conclude that all curves in $\Delta'$ have a property that one of their ends belongs to $E'=F^{-1}(E)$ while the other one to $A=F^{-1}(A_f)$.
    Let us denote by
    \[
    \Delta''=\Gamma(E', A, B)\setminus \Delta'
    \]
    with a restriction that we consider open paths only. By the definition of $A$ we have that $F(A)=A_f$ and, thus, all curves $\ga''$ in $f(\Delta'')$ are such that $f$ does not have an asymptotic value along them. In a consequence all $\ga''$ are nonrectifiable and Corollary~\ref{cor-rect-enough} implies that $\Mod_4(f(\Delta''))=0$. We again apply quasiconformality of $f$ and obtain that $\Mod_4(\Delta'')=0$. Lemma~\ref{lem-non-rect} and the subadditivity of the modulus result in the following observation:
    \begin{align*}
      \Mod_4(E', A, B)=\Mod_4(\Delta'\cup \Delta'')\leq \Mod_4(\Delta')+ \Mod_4(\Delta'')=0.
    \end{align*}
    Define $\Gamma'\subset \Gamma(E', \bd B_{2R}, B_{2R})$ to be the family of those curves joining $E'$ with $\partial B_{2R}$ which are required to pass through $A\subset \bd B_R$. Since $\Gamma(E', A, B_{2R})< \Gamma'$ it holds that
    \[
     \Mod_4(\Gamma')\leq \Mod_4(E', A, B_{2R})=0.
    \]
    Denote by $\Gamma''\subset \Gamma(\bd B_R, \bd B_{2R}, B_{2R})$ the family of curves with one endpoint in $A$, the other one in $\bd B_{2R}$ which entirely lie in the ring $B_{2R}\setminus \overline{B_R}$.

    Then, $\Gamma'=\Gamma''\cup \Gamma(E', A, B_R)$ and $\Gamma''\cap \Gamma(E', A, B_R)=\emptyset$. Therefore,
    \[
     0=\Mod_4(\Gamma')=\Mod_4(\Gamma'')+\Mod_4(E', A, B_R)=\Mod_4(\Gamma'').
    \]
    By Lemma~\ref{mod-cap} this means that ${\rm Cap}_4(A, B_{R})=0$. Lemma 6.15 together with Corollary 4.24 in Bj\"orn--Bj\"orn~\cite{bb} can be applied to the Heisenberg group $\Hei$, due to Theorem~\ref{loew-heis}. From this we conclude that $C_4(A)=0$ as desired.

     \end{proof}

\appendix
\section{Appendix}

 In the section we provide some auxiliary results in the geometry of the Heisenberg group and the modulus of curve families in $\Hei$.

\subsection{Polar coordinates}

 For each $(z,t)\in \Hei$ the curve $\gamma_{(z,t)}(s)=\phi(s,(z,t))$, where
\begin{align}\label{ray}
\phi(s,(z,t))=\delta_s \left( \exp\left(-i t \frac{\log(s)}{|z|^2}\right) z, t\right),\end{align} is a horizontal ray joining $0$ to $(z,t)$. In particular, the parametrization \eqref{ray} has the following properties (see Balogh~\cite{bt}):
\begin{enumerate}
\item $\phi(s,(z,0))=sz$
\item $\|\phi(s,(z,t))\|= \|\gamma_{(z,t)} (s)\|=s \|(z,t)\|$
\item If $\Phi_s(z,t):= \phi(s,(z,t))$, then $\det D \Phi_s(z,t) =s^4$ for $s>0$ and $(z,t)\in \Hei \setminus \mathcal{Z}$ where $\mathcal{Z}$ is the horizontal singular set
$$
 \mathcal{Z}=\{0\} \cup \left\{(z,t) \in \Hei \setminus \{0\}\,:\, \nabla_0 \|(z,t)\|=0\right\}
 $$
 and
 $$
 \nabla_0 \|(z,t)\|= (\tilde X \|(z,t)\|) \tilde X + (\tilde Y \|(z,t)\|) \tilde Y
 $$
 is the horizontal gradient of $\|(z,t)\|$.
\end{enumerate}

We note that if $(z,t) \in B(0,1)$ then property 2 above, shows that $\gamma_{(z,t)}(s)$ is a geodesic joining $0$ to $(z,t)$ and we conclude that $B(0,1)$ is star shaped with respect to $0$. By Theorem 3.7 in \cite{bt}, there exists a unique Radon measure $\sigma$ on $S \setminus \mathcal{Z}$ (for a unit sphere $S=S(0,1)$), such that for $u \in L^1(\Hei)$,
\begin{align} \label{polcords}
 \int_{\Hei} u(z,t) \, d \lambda (z,t) & = \int_{S\setminus \mathcal{Z}} \int_0^\infty u( \phi(s,v)) \, s^{3} ds \, d \sigma(v).
\end{align}
Furthermore, by Proposition 2.18 in \cite{bt} we have $\lambda(\mathcal{Z}) = 0$,
\begin{equation}\label{param-sphere}
 S \setminus \mathcal{Z} =\{ (\sqrt{\cos \alpha} e^{i \theta}, \sin \alpha  ) \,:  \, \alpha \in (-\pi/2,\pi/2), \quad \theta \in [0,2 \pi) \},
\end{equation}
and it follows that $d \sigma =d \alpha d \theta$ (see Example 3.11 in \cite{bt}).

We use the parametrization in \eqref{param-sphere} to define a cone at point $x_0 \in \partial B(p_0,r)$ as follows:
translate and dilate so that $B(0,1)=\delta_{1/r} \circ \tau_{p_0^{-1}} (B(p_0,r))$
 and set $\hat x_0 = \delta_{1/r} \circ \tau_{p_0^{-1}}(x_0) \in \partial B(0,1)$ . Let $\hat x_0 =  (\sqrt{\cos \alpha_0} e^{i \theta_0}, \sin \alpha_0  )$ for some $(\alpha_0, \theta_0)$ and set
$$
C_{\epsilon}(\hat x_0) =\{ \Phi_s(\sqrt{\cos \alpha} e^{i \theta}, \sin \alpha  ) \, : \, \alpha \in (\alpha_0-\epsilon,\alpha_0-\epsilon), \quad \theta \in [0,2 \pi, ),\quad s \in (0,1) \}.
$$
The open cone in $B(p_0,r)$ at $x_0 \in \partial B(p_0,r)$ with angle $\epsilon$ is is the set
\begin{align} \label{cone} C(p_0,r,x_0,\epsilon)=\tau_{p_0} \circ \delta_r (C_{\epsilon}(\hat x_0)).
\end{align}

\subsection{Modulus of curve families in $\Hei$}\label{sec-curves}

The notion of the modulus of curve families is fundamental in the studies of geometry of metric spaces and mappings between domains in metric spaces. In Section~\ref{prel-quasic} we will appeal to the modulus in order to define quasiconformal mappings. For this reason we now recall and briefly discuss modulus of curve families.

We now follow the standard way to define the modulus of curve families, see e.g. Chapter 6 in V\"ais\"al\"a~\cite{va1}. Let $\Gamma$ be a family of curves in a domain $\Om\subset \Hei$. We say that a nonnegative Borel function $\varrho:\Hei\to[0,\infty]$ is \emph{admissible for $\Gamma$} if
\[
 \int_{\gamma} \varrho d l \geq 1,
\]
for every locally rectifiable $\gamma\in \Gamma$. We denote the set of admissible functions by $F(\Gamma)$.

Let $1\leq p <\infty$. Then the \emph{$p$-modulus of curve family $\Gamma$} is defined as follows:
\[
 \modp \Gamma:=\inf_{\varrho\in F(\Gamma)}\int_{\Hei}\varrho^p d\lambda,
\]
where $\lambda$ is $3$-dimensional Lebesgue measure on $\Hei=\R^{3}$. If $F(\Gamma)$ is empty, then by convention we define $\modp \Gamma =\infty$. If $\gamma \in \Gamma$ is a constant curve then the condition $\int_\gamma \rho d l \geq 1$ is not satisfied and the set of admissible functions $F(\Gamma)$ is empty.

From the point of view of relating the $p$-modulus to other geometric data we will often consider curve families joining subsets. If $\Omega \subseteq \mathbb{H}^1 $ is a domain such that $E$ and $F$ are subsets of $\Omega$ then $\Gamma(E,F,\Omega)$ will denote the family of closed rectifiable curves in $\Omega$ which join $E$ and $F$. If $f$ is a homeomorphism of $\Omega$ then we define  $f\Gamma(E,F,\Omega)=\Gamma(f(E),f(F),f(\Omega))$.

The fundamental properties of the p-modulus that we require are summarised in the following lemma (see \cite{hk} section 2.3).

\begin{lem}\label{lem-mod-prop}
 The p-modulus satisfies the following:
\begin{enumerate}
\item  The p-modulus of all curves that are not locally rectifiable is zero.
\item $\modp \emptyset =0$
\item If $\Gamma \subset \Gamma'$ then  $\modp \Gamma \leq \modp\Gamma'$
\item If $\Gamma =\cup_{j=1}^\infty \Gamma_j$ then  $\modp \Gamma \leq \sum_{j=1}^\infty \modp \Gamma_j$.
\item 
If every curve in $\Gamma'$ contains a subcurve in $\Gamma$ then we say that $\Gamma'$ is minorised by $\Gamma$ and write $\Gamma <\Gamma'$. Then, it holds that $\modp \Gamma' \leq \modp\Gamma$.
\end{enumerate}
\end{lem}

In the case $p=4$ we can show that $4$-modulus depends only on rectifiable curves. The following observations are analogues of Corollary 6.11 and Theorem 7.10 in \cite{va1} and are used in the proof of the Tsuji theorem~\ref{tsuji-thm} in Section~\ref{sec-boundary}. Since these results do not appear in the literature we provide their proof.

Given a curve family  $\Gamma$, we denote by $F_r(\Gamma)$, the family of all nonnegative Borel functions $\varrho:\Hei \to \R$  such that $\int_\gamma \varrho dl \geq 1$ for every rectifiable $\gamma \in \Gamma$. Note that  $F(\Gamma) \subseteq F_r(\Gamma)$ with equality when $\Gamma$ consists entirely of closed paths.

\begin{thm} If $\Gamma$ is a curve family in $\Hei$, then
$$\Mod_4\Gamma = \inf_{\varrho\in F_r(\Gamma)} \int_{\Hei}\varrho^4 d\lambda.$$
\end{thm}

\begin{proof}
Since  $F(\Gamma) \subseteq F_r(\Gamma)$ we have $\inf_{\varrho\in F_r(\Gamma)} \int_{\Hei}\varrho^4 d\lambda \leq \Mod_4\Gamma$.
Let
$$
\varrho_1(z,t)=\begin{cases} \frac{1}{||(z,t)|| \log||(z,t)||} &\mbox{if } ||(z,t)|| \geq 2  \\
\quad \quad \quad 1  & \mbox{if } ||(z,t)|| < 2, \end{cases}
$$
then by \eqref{polcords} we have
\begin{align*}
 \int_{\Hei} \varrho_1(z,t)^Q d \lambda(z,t) & =2 \pi^2 \left ( \frac{2^Q}{Q} +\frac{1}{(Q-1) (\log 2)^{Q-1}} \right ).
\end{align*}

Assume $\gamma \in \Gamma$ is locally rectifiable but not rectifiable.  If $\gamma$ is bounded then we have $\varrho_1(g) \geq a>0$ for some $a$ and so
$$ \int_\gamma \varrho_1 d l = \infty.$$ Suppose $\gamma$ is unbounded. Then by a left translation we may assume $\gamma(0)=0$ and that there exists $s_0 \geq 2$ such that $|| \bar \gamma(s)|| \geq 2 $ for all $s \geq s_0$. For each $n>s_0$ let $\bar \gamma_n = \bar \gamma |_{[0,n]}$, then
\begin{align*}
\int_{\bar \gamma_n} \varrho_1(g) d l &= \int_0^{s_0}  \varrho_1(\bar \gamma(s)) \, ds + \int_{s_0}^{n}  \varrho_1(\bar \gamma(s))\, ds\\
& \geq \int_{s_0}^{n}  \varrho_1(\bar \gamma(s))\, ds\\
& \geq \int_{s_0}^{n}  \frac{1}{s \log s} ds\\
& = \log \log n -\log \log s_0
\end{align*} where in the second to last line we have used $||\bar \gamma(s) || \leq s$. Hence we again have that $ \int_\gamma \varrho_1 d l = \infty$.

Let $\varrho \in F_r(\Gamma) $ and set $\varrho_\epsilon =(\varrho^Q+\epsilon^Q \varrho_1^Q)^{1/Q}$, then $\varrho_\epsilon >\varrho$ and
$$ \int_\gamma \varrho_\epsilon d l  \geq \int_\gamma \varrho \, d l \geq 1 $$ for every rectifiable $\gamma \in \Gamma$. If $\gamma \in \Gamma$ is not rectifiable then
$$ \int_\gamma \varrho_\epsilon d l  \geq \epsilon \int_\gamma \varrho_1 \, d l = \infty. $$ It follows that  $\varrho_\epsilon \in F(\Gamma)$ and $$\Mod_4 \Gamma \leq \int_{\Hei} \varrho_\epsilon^4 \, d \lambda =  \int_{\Hei} \varrho^4 d \lambda + \epsilon^4\int_{\Hei} \varrho_1^4 \, d \lambda.$$ Since $\epsilon >0$ and $\varrho$ are arbitrary we conclude that $\Mod_4 \Gamma \leq \inf_{\varrho\in F_r(\Gamma)} \int_{\Hei}\varrho^4 d\lambda$
\end{proof}

\begin{cor}\label{cor-rect-enough}
  If $\Gamma_r$ is the family of all rectifiable curves in $\Gamma$, then $\Mod_4\Gamma = \Mod_4 \Gamma_r$. In particular, the family of all non-rectifiable curves in $\Hei$ has zero $4$-modulus.
\end{cor}

Let $\Gamma_0(E, F, \Om)$ denote the family of all curves $\gamma$ in $\Om$ with the property that the closure of the trace of $\gamma$ has nonempty intersection with both $E$ and $F$.

\begin{lem}\label{lem-non-rect} Let $\Gamma_0=\Gamma_0(E, F, \Om)$ denote the family of all curves $\gamma$ in $\Om$ with the property that the closure of the trace of $\gamma$ has nonempty intersection with both $E$ and $F$. If $\Gamma =\Gamma(E, F, \Om)$ then
 \[
  \Mod_4(\Gamma_0)=  \Mod_4(\Gamma).
 \]
 \end{lem}
\begin{proof}
 Since  $\Gamma$ is minorised by  $\Gamma_0$ we have $\Mod_4 \Gamma \leq \Mod_4 \Gamma_0$. In order to prove the reverse inequality it suffices to prove that $F(\Gamma) \subset F_r(\Gamma_0)$.

Assume that $\varrho \in F(\Gamma)$ and that $\gamma $ is a rectifiable path in $\Gamma_0$. If $\gamma^*$ denotes the closed extension of $\gamma$ given by Theorem \ref{curve-ext}, then the trace of $\gamma^*$ meets both $E$ and $F$. In  particular we may assume there exists $t_1 \leq t_2$ such that  $\gamma^*(t_1) \in E$ and $\gamma^*(t_2) \in F$. It follows that the curve $\beta =\gamma^*|_{[t_1,t_2]}$ belongs to $\Gamma$ and
$$ \int_\gamma \varrho dl = \int_{\gamma^*} \varrho dl \geq \int_{\beta} \varrho dl \geq 1. $$
We conclude that $\varrho \in F_r(\Gamma_0)$.
\end{proof}

\setmarginsrb{20mm}{15mm}{20mm}{15mm}{10mm}{10mm}{10mm}{10mm}

\end{document}